\documentclass{article}

\usepackage{amssymb}
\usepackage{amsthm}                
\usepackage{amstext}
\usepackage{amsbsy}
\usepackage{amscd}
\usepackage{enumerate}
\usepackage{subfigure}
\usepackage[noadjust]{cite}
\usepackage[bookmarks=true,colorlinks=true, pdfborder={0 0 0}, linktocpage, linkcolor=red, citecolor=blue]{hyperref}
\usepackage{graphicx} 
\usepackage{mathtools}
\usepackage{pict2e}
\usepackage[usenames,dvipsnames]{color}

\newtheorem{thm}{Theorem}[section]

\newtheorem{lem}[thm]{Lemma}
\newtheorem{prop}[thm]{Proposition}
\newtheorem{conj}[thm]{Conjecture}

\theoremstyle{remark}
\newtheorem{rem}[thm]{Remark}

\theoremstyle{definition}

\theoremstyle{plain}
\newtheorem{claim}[thm]{Claim}

\newcommand{\norm}[1]{\left\Vert#1\right\Vert}
\newcommand{\abs}[1]{\left\vert#1\right\vert}
\newcommand{\set}[1]{\left\{#1\right\}}
\newcommand{\R}{\mathbb{R}}

\newcommand{\N}{\mathbb{N}}
\newcommand{\T}{\mathbb{T}}
\newcommand{\Z}{\mathbb{Z}}

\renewcommand{\dim}{\mathbf{dim}}
\newcommand{\len}{\mathbf{len}}
\newcommand{\dist}{\mathbf{dist}}
\newcommand{\proj}[2]{\mathbf{Proj}_{#1}\left(#2\right)}
\newcommand{\hdim}{\dim_\mathrm{H}}
\newcommand{\lhdim}{\hdim^\loc}
\newcommand{\loc}{\mathrm{loc}}
\newcommand{\Tr}{\mathrm{Tr}\text{\hspace{1mm}}}
\newcommand{\bdim}{\dim_\mathrm{B}}
\newcommand{\lbdim}{\bdim^\loc}

\newcommand{\mf}[1]{\mathfrak{#1}}

\newcommand{\cites}{\cite}

\begin{document}

\title{Spectral analysis of tridiagonal
Fibonacci Hamiltonians}

\author{William N. Yessen$^*$\\
Department of Mathematics\\
University of California, Irvine\\
Irvine, CA 92697\\
wyessen@math.uci.edu}

\date{\today}

\setcounter{tocdepth}{1}

\maketitle

\begin{abstract}

We consider a family of discrete Jacobi operators on the one-dimensional integer lattice, with the diagonal and the off-diagonal entries given by two sequences generated by the Fibonacci substitution on two letters. We show that the spectrum is a Cantor set of zero Lebesgue measure, and discuss its fractal structure and Hausdorff dimension. We also extend some known results on the diagonal and the off-diagonal Fibonacci Hamiltonians.

\vspace{2mm}
2010 Mathematics Subject Classification: 47B36, 82B44;

\vspace{2mm}
Keywords: spectral theory, quasiperiodicity, Jacobi operators, Fibonacci Hamiltonians, trace maps.

\vspace{5mm}
$^*$The author was supported by the NSF grant DMS-0901627, PI: A. Gorodetski and the NSF grant IIS-1018433, PI: M. Welling and Co-PI: A. Gorodetski

\end{abstract}

\tableofcontents

\section{Introduction}


Partly due to the choice of the models in the original papers \cites{Kohmoto1983, Ostlund1983}, until quite recently, the mathematical literature on the Fibonacci operators had been focused exclusively on the diagonal model (see surveys \cites{Damanik2005, Damanik2000b,Suto1995}). Recently in \cite[Appendix A]{Damanik2010} D. Damanik and A. Gorodetski, and also J. M. Dahl in \cite{Dahl2010}, investigated the off-diagonal model. This model has been the object of interest in a number of physics papers (see, for example, \cites{Mandel2008, Mandel2006, Kohmoto1987b, Rudinger1998, Velhinho2000, You1991}). 

Quasi-periodicity has also been considered, as early as 1987, in a widely studied model of magnetism: the Ising model, both quantum and classical; numerous numerical and some analytic results were obtained (see \cites{Tsunetsugu1987, Benza1989, Ceccatto1989, Hermisson1997, Doria1988, Benza1990, You1990, Tong1997} and references therein). Recently the author investigated some properties of these models in \cite{Yessen2011}. The following problem was motivated as a result of this investigation. What can be said about the spectrum and spectral type of the tridiagonal Fibonacci Hamiltonian? The aim of this paper is to investigate spectral properties of such operators.

In general one would hope to parallel the development for the diagonal and the off-diagonal cases; however, a fundamental difference presents some technical difficulties: in the application of the trace map one finds that the constant of motion (the so-called \textit{Fricke-Vogt invariant}), unlike in the diagonal and the off-diagonal cases, is not energy-independent. The main tool in the investigation of the diagonal and the off-diagonal operators has been hyperbolicity of the trace map when restricted to a constant of motion. While this technique will not apply in our case verbatim, motivated by it, and in part based on it, we employ some other tools to combat the aforementioned difficulties.

\section{The model and main results}\label{part_one}


\subsection{The model}\label{the_model}


Let $\mathcal{A} = \set{a,b}$; $\mathcal{A}^*$ denotes the set of finite words over $\mathcal{A}$. The Fibonacci substitution $S: \mathcal{A}\rightarrow\mathcal{A}^*$ is defined by $S: a\mapsto ab$, $S: b\mapsto a$. We formally extend the map $S$ to $\mathcal{A}^*$ and $\mathcal{A}^{\N,\Z}$ by
\begin{align*}
S: \alpha_1\alpha_2\cdots\alpha_k\mapsto S(\alpha_1)S(\alpha_2)\cdots S(\alpha_k)
\end{align*}
and
\begin{align*}
S: \cdots\alpha_1\alpha_2\cdots\mapsto \cdots S(\alpha_1)S(\alpha_2)\cdots.
\end{align*}
There exists a unique \textit{substitution sequence} $u\in\mathcal{A}^\N$ with the following properties \cite{Queffelec2010}:
\begin{align}\label{eq_u_prop}
&u_1\cdots u_{F_k} = S^{k-1}(a),\text{\hspace{2mm}}k\geq 2;\notag\\
&S(u) = u;\\
&u_1\cdots u_{F_{k+2}} = u_1\cdots u_{F_{k + 1}}u_1\cdots u_{F_k}\notag,
\end{align}
where $\set{F_k}_{k\in\N}$ is the sequence of Fibonacci numbers: $F_0 = F_1 = 1;\text{\hspace{2mm}}F_{k\geq 2} = F_{k-1} + F_{k-2}$.
From now on we reserve the notation $u$ for this specific sequence. 

Let $\hat{u}$ denote an arbitrary extension of $u$ to a two-sided sequence in $\mathcal{A}^\Z$. Equip $\mathcal{A}$ with the discrete topology and $\mathcal{A}^{\N,\Z}$ with the corresponding product topology. Define
\begin{align*}
\Omega = \set{\omega\in\mathcal{A}^\Z: \omega = \lim_{i\rightarrow\infty}T^{n_i}(\hat{u}), n_i\uparrow\infty},
\end{align*}
where $T: \mathcal{A}^\Z\rightarrow\mathcal{A}^\Z$ is the left shift: for $v\in\mathcal{A}^\Z,\text{\hspace{2mm}}[T(v)]_n = v_{n+1}$. The \textit{hull} $\Omega$ is compact and $T$-invariant, and $T$ is continuous.
Now to each $\omega\in\Omega$ we associate a Jacobi operator.

For every $\omega\in\Omega$, we define the \textit{Fibonacci Jacobi operator} or \textit{tridiagonal Fibonacci Hamiltonian}, $H_\omega$, on $l^2(\Z)$ as follows. Let $p, q: \mathcal{A}\rightarrow\R$. We allow only nonzero values for $p$.
\begin{align}\label{eq_models}
(H_\omega\phi)_n = p(\omega_{n})\phi_{n-1} + p(\omega_{n+1})\phi_{n+1} + q(\omega_n)\phi_n.
\end{align}
When $q\equiv 0$, we call $H$ the \textit{diagonal model} and when $p\equiv 0$, we call $H$ the \textit{off-diagonal model}. Clearly these two models are special cases of the tridiagonal Hamiltonian.

We single out a special $\omega_s\in\Omega$, defined as follows. Notice that $ba$ occurs in $u$ and that $S^2(a) = aba$ begins with $a$ and $S^2(b) = ab$ ends with $b$. Thus, iterating $S^2$ on $b|a$, where $|$ denotes the origin, we obtain as a limit a two-sided infinite sequence $\omega_s$ in $\Omega$.
%
%
The sequence $\omega_s$ has the following properties.
\begin{align}\label{eq_extended_u}
\text{\hspace{2mm}}[\omega_s]_{k\geq 1} = u_k;\text{\hspace{2mm}}[\omega_s]_{-k}=u_{k-1}\text{\hspace{2mm}for all\hspace{2mm}}k\geq 2.
\end{align}

\subsection{Main results}

From now on the spectrum of an operator $H$ will be denoted by $\sigma(H)$. The operators in \eqref{eq_models} can be first scaled by $p(a)$ and then shifted by $-q(a)/p(a)$ while preserving the spectrum. So without loss of generality, we may assume that $p(a) = 1$ and $q(a) = 0$. We represent $p, q$ in compact vector notation $(\mf{p},\mf{q})$, where $p(b) = \mf{p}$ and $q(b) = \mf{q}$.
\begin{thm}\label{thm_first_result}
There exists $\Sigma_{(\mf{p}, \mf{q})}\subset\R$, such that for all $\omega\in\Omega$, $\sigma(H_\omega) = \Sigma_{(\mf{p}, \mf{q})}$. If $(\mf{p},\mf{q})\neq (1,0)$, then $\Sigma_{(\mf{p}, \mf{q})}$ is a Cantor set of zero Lebesgue measure; it is purely singular continuous.
\end{thm}
\begin{rem} By a Cantor set we mean a (nonempty) compact totally disconnected set with no isolated points.
\end{rem}
%
%
%
%
%
\begin{figure}[t]
\centering
\includegraphics[scale=.5]{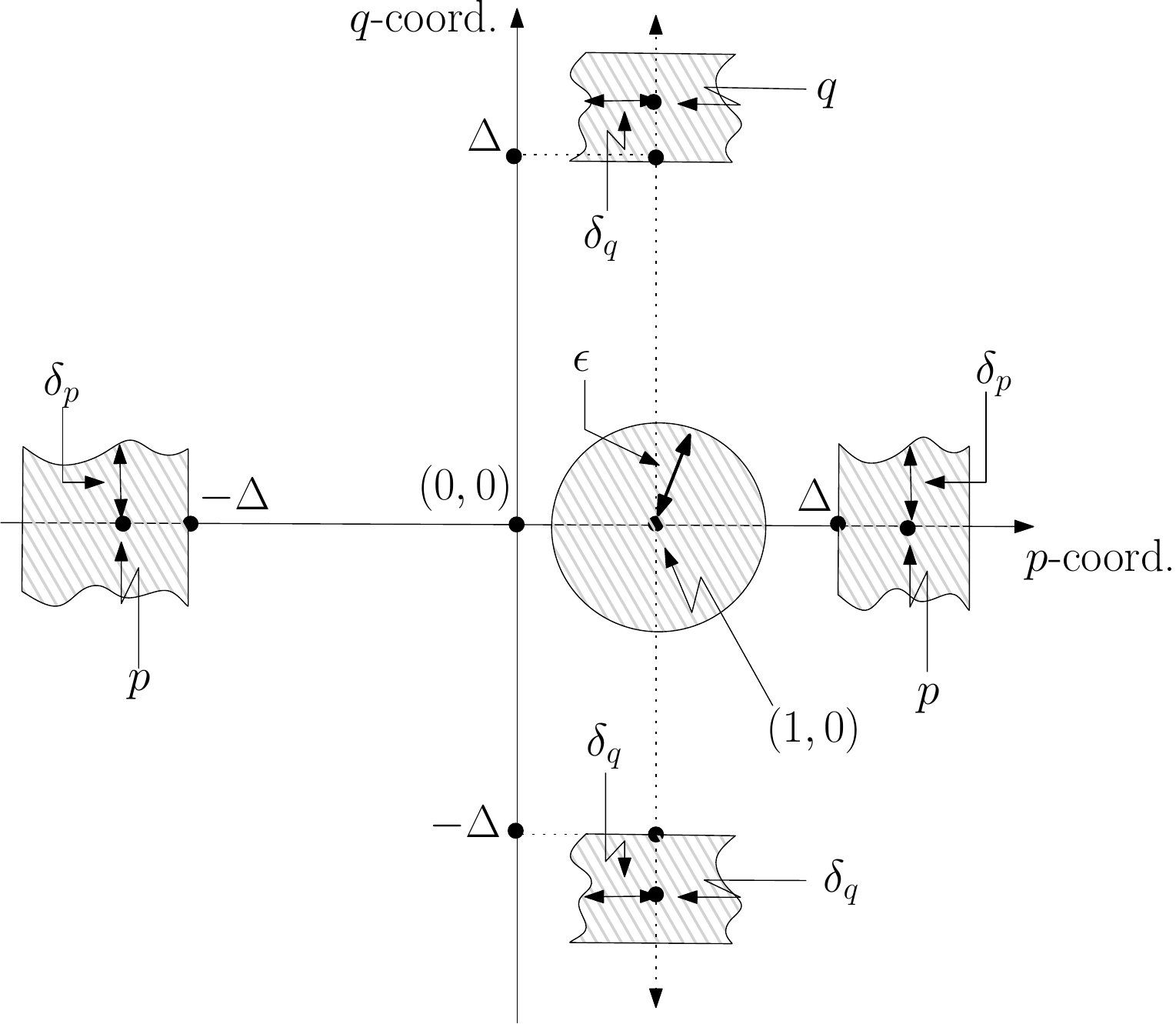}
\caption{}
\label{fig:box-region}
\end{figure}
We write simply $H$ for $H_{\omega_s}$. In what follows, the Hausdorff dimension of $A\subset\R$ is denoted by $\hdim(A)$. The local Hausdorff dimension of $A$ at $a\in A$ is defined as
\begin{align*}
\lhdim(A, a) := \lim_{\epsilon\rightarrow 0}\hdim(A\cap (a - \epsilon, a + \epsilon)).
\end{align*}
We denote by $\bdim(A)$ the box-counting dimension of $A$, and define $\lbdim(A)$ similarly to $\lhdim(A)$.

Our next results is the following theorem that describes  the fractal structure of the spectrum. 

\begin{thm}\label{thm_main_result1} For all $(\mf{p},\mf{q})\neq (1,0)$, the spectrum $\Sigma_{(\mf{p},\mf{q})}$ is a multifractal; more precisely, the following holds.

\begin{enumerate}[i.]
\item $\lhdim(\Sigma_{(\mf{p},\mf{q})}, a)$, as a function of $a\in \Sigma_{(\mf{p},\mf{q})}$, is continuous; It is constant in the diagonal and the off-diagonal cases, and nonconstant otherwise;
\item There exists nonempty $\mf{N}\subset\R^2$ of Lebesgue measure zero, such that the following holds.
\begin{enumerate}[(a)]
\item For all $(\mf{p},\mf{q})\notin \mf{N}$, we have $0<\lhdim(\Sigma_{(\mf{p},\mf{q})}, a)<1$ for all $a\in\Sigma_{(\mf{p},\mf{q})}$; hence we have $0<\hdim(\Sigma_{(\mf{p},\mf{q})})<1$;
\item for $(\mf{p},\mf{q})\in \mf{N}$, $0 < \lhdim(\Sigma_{(\mf{p},\mf{q})}, a) < 1$ for all $a\in \Sigma_{(\mf{p},\mf{q})}$ away from the lower and upper boundary points of the spectrum, and $\hdim(\Sigma_{(\mf{p},\mf{q})}) = 1$. In fact, the dimension accumulates at one of the two ends of the spectrum.
\end{enumerate}
\item $\lim_{(\mf{p},\mf{q})\rightarrow(1,0)}{\hdim(\Sigma_{(\mf{p},\mf{q})})} = 1$. In fact, the Hausdorff dimension of the spectrum is a continuous function of the parameters;
\item $\hdim(\Sigma_{(\mf{p}, 0)})$ and $\hdim(\Sigma_{(1, \mf{q})})$ depend analytically on $\mf{p}$ and $\mf{q}$, respectively;
\end{enumerate}
\end{thm}
\begin{rem} We conjecture a stronger result in Section \ref{conclusion}. We also mention that \textit{ii-(a)} and \textit{iv} are extensions of results on the diagonal and the off-diagonal model; indeed, previous results relied on transversality arguments (see below), but transversality is still not known for some values of parameters $\mf{p}$ and $\mf{q}$ (see Section \ref{conclusion}). Notice also that unlike in the previously considered diagonal and off-diagonal models, in the tridiagonal model the spectrum may have full Hausdorff dimension even in the non-pure regime (i.e., $(\mf{p},\mf{q})\neq (1,0)$).
\end{rem}

Existence of box-counting dimension and, if it exists, whether it coincides with the Hausdorff dimension, is of interest. The next theorem provides a partial answer in this direction. Indeed, we prove that for all parameters $(\mf{p},\mf{q})$ in a certain region in $\R^2$ (the shaded regions in Figure \ref{fig:box-region}), the box-counting dimension of $\Sigma_{(\mf{p},\mf{q})}$ exists and coincides with the Hausdorff dimension (see, however, Section \ref{conclusion}). 

\begin{thm}\label{thm_main1}
The following statements hold.

\begin{enumerate}[i.]
\item There exists $\epsilon > 0$ such that for all $(\mf{p},\mf{q})$ satisfying $\norm{(1,0) - (\mf{p},\mf{q})}<\epsilon$, the box-counting dimension of $\Sigma_{(\mf{p},\mf{q})}$ exists and coincides with the Hausdorff dimension;

\item There exists $\Delta > 0$, such that for all $\abs{\mf{p}} \geq \Delta$ there exists $\delta_\mf{p}>0$, such that for all $\mf{q}$ satisfying $\abs{\mf{q}}<\delta_\mf{p}$, the box-counting dimension of $\Sigma_{(\mf{p},\mf{q})}$ exists and coincides with the Hausdorff dimension.

\item There exists $\Delta > 0$ such that for all $\abs{\mf{q}}\geq\Delta$ there exists $\delta_\mf{q} > 0$, such that for all $\mf{p}$ satisfying $\abs{\mf{p}}<\delta_\mf{q}$, the box-counting dimension of $\Sigma_{(\mf{p},\mf{q})}$ exists and coincides with the Hausdorff dimension.
\end{enumerate}
\end{thm}

In the statement of the next theorem, denote the density of states for the operator $H_{(\mf{p},\mf{q})}$ by $\mathcal{N}$ and the corresponding measure by $d\mathcal{N}$ (for definitions, properties and examples, see, for example, \cite[Chapter 5]{Teschl1999}). Of course, $\mathcal{N}$, and consequently $d\mathcal{N}$, depend on $(\mf{p},\mf{q})$. We quickly recall that $d\mathcal{N}$ is a non-atomic Borel probability measure  on $\R$ whose topological support is the spectrum $\Sigma_{(\mf{p},\mf{q})}$.

The next theorem states that the point-wise dimension of $d\mathcal{N}$ exists $d\mathcal{N}$-almost everywhere, but may depend on the point, unlike in the diagonal case (compare Theorem \ref{thm_main2} with the results of \cite{Damanik2011}).

\begin{thm}\label{thm_main2}
For all $(\mf{p},\mf{q})\in \R^2$, there exists $\mf{V}_{(\mf{p},\mf{q})}\subset \R$ of full $d\mathcal{N}$-measure, such that for all $E\in \mf{V}_{(\mf{p},\mf{q})}$ we have
\begin{align}\label{eq:thm_main2-1}
\lim_{\epsilon\downarrow 0}\frac{\log\mathcal{N}(E-\epsilon, E+\epsilon)}{\log\epsilon} = d_{(\mf{p},\mf{q})}(E)\in\R,
\end{align}
$d_{(\mf{p},\mf{q})}(E) > 0$. Moreover, if $(\mf{p},\mf{q})\neq (1,0)$, then 
\begin{align}\label{eq:thm_main2-3}
d_{(\mf{p},\mf{q})}(E) < \lhdim(\Sigma_{(\mf{p},\mf{q})}, E).
\end{align}
Also, 
\begin{align}\label{eq:thm_main2-2}
\lim_{(\mf{p},\mf{q})\rightarrow (1,0)}\sup_{E\in\mf{V}_{(\mf{p},\mf{q})}}\set{d_{(\mf{p},\mf{q})}(E)} = \lim_{(\mf{p},\mf{q})\rightarrow (1,0)}\inf_{E\in\mf{V}_{(\mf{p},\mf{q})}}\set{d_{(\mf{p},\mf{q})}(E)} = 1.
\end{align}
\end{thm}

\section{Proof of main results}\label{sec_proves}

Assume, unless stated otherwise, that $(\mf{p},\mf{q}) \neq (1,0)$. Let $\widetilde{\omega}_k$ be a periodic word of period $F_k$ with unit cell $[\omega_s]_1\cdots[\omega_s]_{F_k}$. Let
\begin{align*}
(\widetilde{H}^k\phi)_n = p([\widetilde{\omega}_k]_{n})\phi_{n-1} + p([\widetilde{\omega}_k]_{n+1})\phi_{n+1} + q([\widetilde{\omega}_k]_{n})\phi_{n}.
\end{align*}
If $\theta(\lambda)\in\R^\Z$ satisfies
\begin{align}\label{eq_diff}
\widetilde{H}^k\theta(\lambda) = \lambda\theta(\lambda),
\end{align}
then for all $n\in\Z$,
\begin{align}\label{eq_recurrence_ahaml}
p([\widetilde{\omega}_k]_{n+1})\theta_{n+1}(\lambda) = (\lambda - q([\widetilde{\omega}_k]_n))\theta_n(\lambda) - p([\widetilde{\omega}_k]_{n})\theta_{n-1}(\lambda).
\end{align}
Take $\psi(\lambda), \phi(\lambda)\in\R^\Z$, with $\phi_0 = \psi_{-1} = 1$, $\phi_{-1} = \psi_0 = 1$, satisfying \eqref{eq_diff}. By Floquet theory \cite{Toda1981} ,
\begin{align}\label{the_spectrum}
\sigma(\widetilde{H}^k) = \sigma_k:=\set{\lambda: \frac{1}{2}\abs{\phi_{F_k}(\lambda) + \psi_{F_k-1}(\lambda)}\leq 1}.
\end{align}
We write $p_{k,n}$ for $p([\widetilde{\omega}_k]_n)$; similarly for $q$. Define
\begin{align}\label{eq_trans_mat}
M_n(\lambda):=
\frac{1}{p_{k,n}}\begin{pmatrix}
\lambda - q_{k,n} & -p_{k,n-1}\\
p_{k,n} & 0
\end{pmatrix};\text{\hspace{2mm}}
T_n(\lambda):=
\frac{1}{p_{k,n}}\begin{pmatrix}
\lambda - q_{k,n} & -1\\
p_{k,n}^2 & 0
\end{pmatrix}
\end{align}
and let $\Theta_n = (\theta_n, p_{k,n}\theta_{n-1})^\mathrm{T}$. By \eqref{eq_recurrence_ahaml}, $\theta$ satisfies \eqref{eq_diff} if and only if
\begin{align}\label{eq_theta_trans}
\begin{pmatrix}
\theta_n\\
\theta_{n-1}
\end{pmatrix} = M_n
\begin{pmatrix}
\theta_{n-1}\\
\theta_{n-2}
\end{pmatrix}\iff
\Theta_n = T_n\Theta_{n-1}\text{\hspace{2mm}for all\hspace{2mm}}n\in\Z.
\end{align}
Define
\begin{align*}
\widehat{T}_k(\lambda) = T_{F_k}(\lambda)\times\cdots\times T_1(\lambda).
\end{align*}
From \eqref{eq_theta_trans} we have $\Theta_{F_k} = \widehat{T}_k\Theta_0$; hence using $\phi$ and $\psi$ in place of $\theta$ we get $\phi_{F_k} = [\widehat{T}_k]_{11}$ and $p_{k,F_k}\psi_{F_k - 1} = p_{k,0}[\widehat{T}_k]_{22}$. Since $\widetilde{\omega}_k$ is $F_k$-periodic, $p_{k,F_k} = p_{k,0}$, so
\begin{align}\label{eq_disc_to_trace}
\frac{1}{2}\abs{\phi_{F_k}(\lambda) + \psi_{F_k - 1}(\lambda)} = \frac{1}{2}\abs{\Tr \widehat{T}_k(\lambda)}.
\end{align}

\subsection{Proof of Theorem \ref{thm_first_result}}\label{proof_thm_first}

Let $\Sigma_{(\mf{p},\mf{q})}$ denote $\sigma(H_{\omega_s})$. It is known that $(\Omega, T)$ is topologically minimal, hence for all $\omega\in\Omega$, $\sigma(H_\omega) = \Sigma_{(\mf{p},\mf{q})}$ (see, for example, \cite{Damanik2000}).

Since $\widehat{T}_k$ is unimodular and, by \eqref{eq_u_prop}, $\widehat{T}_{k+2} = \widehat{T}_{k+1}\widehat{T}_k$, we have, with $2x_k = \Tr \widehat{T}_k$,
\begin{align*}
(x_{k+3}, x_{k+2}, x_{k+1}) = f(x_{k+2}, x_{k+1}, x_k),
\end{align*}
where $f(x,y,z) = (2xy - z, x, y)$ is the \textit{Fibonacci trace map} (for a survey, see \cite{Baake1999} and references therein). The initial condition $(x_3, x_2, x_1)$ is rather complicated. For a simpler expression, we take (we omit calculations)
\begin{align}\label{gamma_defn}
\gamma(\lambda):= (x_1, x_0, x_{-1}) = f^{-2}(x_3,x_2,x_1)= \left(\frac{\lambda - \mf{q}}{2}, \frac{\lambda}{2\mf{p}}, \frac{1 + \mf{p}^2}{2\mf{p}}\right),
\end{align}
where $f^{-1}(x,y,z) = (y, z, 2yz - x)$ is the inverse of $f$ (compare with the initial conditions in, for example, \cite{Damanik2009} and in \cite[Appendix A]{Damanik2010}). We write $\gamma_{(\mf{p},\mf{q})}$ to emphasize dependence on $(\mf{p},\mf{q})$ when necessary.

Fix $C > \abs{(1 + \mf{p}^2)/2\mf{p}}\geq 1$ and for $k\geq -1$ define
\begin{align*}
\widehat{\sigma}_k = \set{\lambda: \frac{1}{2}\abs{x_k}\leq C}.
\end{align*}
These sets are closed and $\widehat{\sigma}_k\cup\widehat{\sigma}_{k+1}\supseteq \widehat{\sigma}_{k+1}\cup\widehat{\sigma}_{k+2}$. Moreover, for any $l\geq -1$,
\begin{align}\label{eq_b_infty}
\bigcap_{k\geq l}\widehat{\sigma}_k\cup\widehat{\sigma}_{k+1} = B_\infty:=\set{\lambda: \mathcal{O}_f^+(\gamma(\lambda))\text{\hspace{2mm}is bounded}},
\end{align}
where $\mathcal{O}_f^+(\mathbf{x}) = \set{\mathbf{x},f(\mathbf{x}), f^2(\mathbf{x}),\dots}$ is the positive semi-orbit of $\mathbf{x}$ under $f$ (see \cite[Proposition 3.1]{Yessen2011}, which is a slight extension of \cite[Proposition 5.2]{Damanik2005}). Since $\widetilde{H}^k\xrightarrow[k\rightarrow\infty]{}H$ strongly, combining \eqref{eq_b_infty}, \eqref{eq_disc_to_trace} and \eqref{the_spectrum}, we get
\begin{align*}
\Sigma_{(\mf{p},\mf{q})}\subset\bigcap_{l\geq 1}\overline{\bigcup_{k\geq l}\sigma_k}\subset \bigcap_{k\geq 1}\widehat{\sigma}_k\cup\widehat{\sigma}_{k+1} = B_\infty.
\end{align*}
Since $\set{p_{k,n}}_{k,n\in\N}$ is uniformly bounded away from zero and infinity and $\omega_s$ satisfies \eqref{eq_extended_u}, the argument in \cite{Suto1987} applies and gives $B_\infty\subseteq \Sigma_{(\mf{p},\mf{q})}$. Hence
\begin{align}\label{eq:dyn-op-spec}
B_\infty = \Sigma_{(\mf{p},\mf{q})}.
\end{align}
(See also Remark \ref{rem:alt-proof} below for an outline of an alternative proof of \eqref{eq:dyn-op-spec}).

Define
\begin{align*}
\mathcal{Z} = \set{\lambda: \lim_{k\rightarrow\infty}\frac{1}{k}\log\norm{\widehat{T}_k(\lambda)} = 0}.
\end{align*}
By Kotani theory (see \cites{Kotani1989, Damanik2007b}, and \cite{Remling2011} for extension to Jacobi operators), $\mathcal{Z}$ has zero Lebesgue measure, and by \cite{Iochum1991}, $B_\infty\subseteq\mathcal{Z}$ (this also follows from an earlier work by A. S\"ut\H{o} -- see \cite{Suto1989} -- and a later (and more general) work of D. Damanik and D. Lenz in \cite{Damanik1999}). Hence $\Sigma_{(\mf{p},\mf{q})}$ has zero Lebesgue measure.

The argument in \cite[Section A.3]{Damanik2010}, without modification, shows that for all $\omega\in\Omega$ $\sigma(H_\omega)$ is purely singular continuous. So $\Sigma_{(\mf{p},\mf{q})}$ contains no isolated points, is compact and has zero Lebesgue measure. Thus it is a Cantor set. This completes the proof.

\begin{rem}\label{rem:alt-proof}
An alternative proof of \eqref{eq:dyn-op-spec} can be given as follows. Using the results of \cite{Avron1990}, we get convergence in Hausdorff metric of the sequence of spectra of periodic approximations, $\set{\sigma_k}$, to the spectrum of the limit quasi-periodic operator. On the other hand, \cite[Theorem 2.1-i]{Yessen2011} shows convergence of $\set{\sigma_k}$ to $B_\infty$. One only needs to note that \cite[Theorem 2.1-i]{Yessen2011} relies on transversality (see Section \ref{proof_i} below), which, as discussed below, we have everywhere except possibly at finitely many points (which does not affect the conclusion of \cite[Theorem 2.1-i]{Yessen2011}).
\end{rem}

\subsection{Proof of Theorem \ref{thm_main_result1}}\label{proof_thm_second}

For the necessary notions from hyperbolic and partially hyperbolic dynamics, see a brief outline in \cite[Appendix B]{Yessen2011}, and \cites{Hirsch1968, Hirsch1970, Hirsch1977, Hasselblatt2002b, Hasselblatt2002} for details.

Define the so-called \textit{Fricke-Vogt invariant} by
\begin{align*}
I(x,y,z) := x^2 + y^2 + z^2 - 2xyz - 1,
\end{align*}
and the corresponding level sets
\begin{align*}
S_V := \set{(x,y,z)\in\R^3: I_V(x,y,z) - V = 0}
\end{align*}
(see Figure \ref{part2_fig1}).
%
%
%
%
%
\begin{figure}[t]
\centering
 \subfigure[$V = 0.0001$]{
\includegraphics[scale=.25]{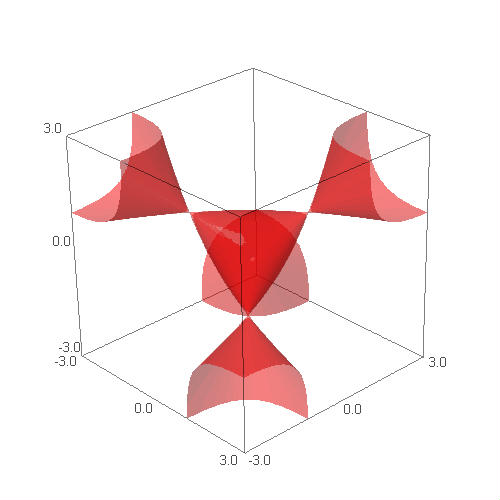}}
 \subfigure[$V = 0.01$]{
\includegraphics[scale=.25]{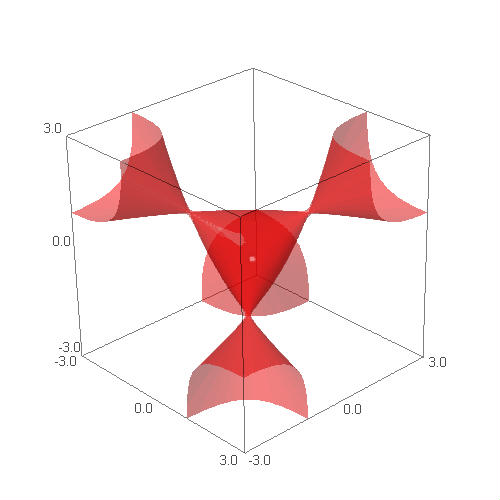}}
\\
 \subfigure[$V = 0.05$]{
\includegraphics[scale=.25]{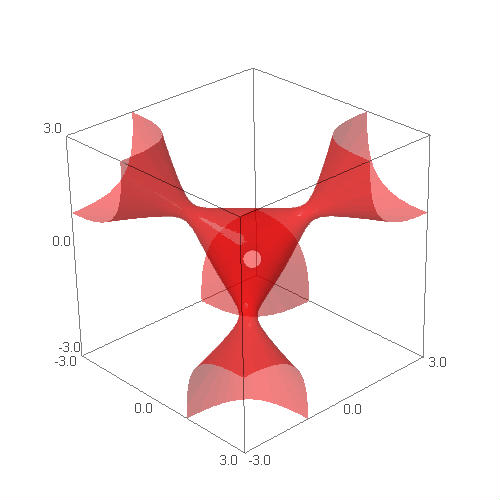}}
 \subfigure[$V = 1$]{
\includegraphics[scale=.25]{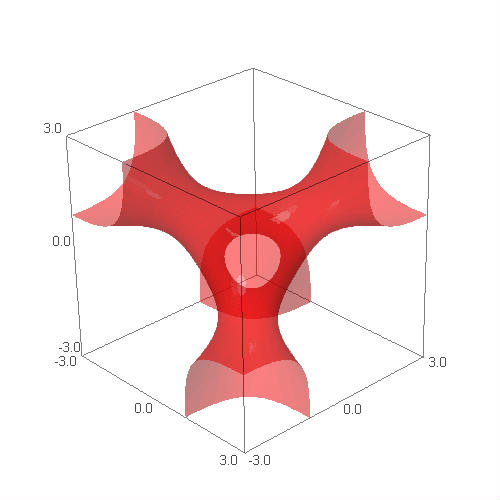}}
\caption{Invariant surfaces $S_V$ for four values of $V$.}
\label{part2_fig1}
\end{figure}
We're interested in $S_{V > 0}$. In this case $S_V$ is a non-compact, connected analytic two-dimensional submanifold of $\R^3$. We have $I_V\circ f = I_V$, consequently $f(S_V) = S_V$. We write $f_V$ for $f|_{S_V}$. The nonwandering set $\Omega_V$ for $f_V$ on $S_V$ is compact $f_V$-invariant locally maximal transitive hyperbolic set (see \cites{Casdagli1986, Cantat2009, Damanik2009}). Consequently, for $x\in S_V$, $\mathcal{O}_{f_V}^+(x)$ is bounded if and only if there exists $y\in\Omega_V$ with $x\in W^\mathrm{s}(y)$, the stable manifold at $y$ (this follows from general principles). There exists a family  $\mathcal{W}^\mathrm{s}$ of smooth two-dimensional injectively immersed pair-wise disjoint submanifolds of $\R^3$, called the \textit{center-stable manifolds} and denoted $W^\mathrm{cs}$, such that
\begin{align*}
\bigcup_{V > 0}\cup_{y\in\Omega_V}W^\mathrm{s}(y) = \bigcup_{W^\mathrm{cs}\in\mathcal{W}^\mathrm{s}}W^\mathrm{cs}
\end{align*}
(see \cite[Proposition 3.9]{Yessen2011}). It follows that for $x\in S_V$, $\mathcal{O}_f^+(x)$ is bounded if and only if $x\in W^\mathrm{cs}$ for some $W^\mathrm{cs}\in\mathcal{W}^s$.

\subsubsection{Proof of i}\label{proof_i}

In the proof below, isolation of tangential intersections (if such exist) was suggested by A. Gorodetski, and the use of \cite[Lemma 6.4]{Bedford1993} was suggested by S. Cantat.

We have
\begin{align}\label{eq_fv_value}
I\circ\gamma(\lambda) = \frac{\lambda\mf{q}(1-\mf{p}^2) + \mf{q}^2\mf{p}^2 + (\mf{p}^2 - 1)^2}{4\mf{p}^2},
\end{align}
which is $\lambda$-dependent (compare with \cites{Damanik2009} and \cite[Appendix A]{Damanik2010}).  Denote by $\gamma^*$ the image of $\gamma$. Since $\gamma^*\subset\set{z = \frac{1 + \mf{p}^2}{2\mf{p}}}$, which is away from the unit cube $\set{(x,y,z): \abs{x},\abs{y},\abs{z}\leq 1}$ when $\mf{p}\neq 1$, for all $\lambda$ with $I\circ\gamma(\lambda) < 0$ (which can only happen when $\mf{p}\neq 1$), $\mathcal{O}_f^+(\gamma(\lambda))$ escapes to infinity (see \cite{Roberts1996}), and these points do not interest us. Application of \cite[Section 3]{Kadanoff1984} with the initial conditions \eqref{gamma_defn} in mind gives similar result for all $\lambda$ sufficiently large. Thus we restrict our attention to a compact line segment along $\gamma^*$, which we denote by $\overline{\gamma^*}$, and which lies entirely in $\bigcup_{V > 0}S_V$.

Take $m\in\overline{\gamma^*}$ whose forward orbit is bounded. Let $U_m$ be a small neighborhood of $m$ in $\R^3$. Pick a plane $\Pi_m$ containing $\overline{\gamma^*}$ and transversal at $m$ to the center-stable manifold containing $m$. Since $f_V$ is analytic and depends analytically on $V$, the center-stable manifolds are analytic (for a detailed proof in the case of Anosov diffeomorphisms, see \cite[Theorem 1.4]{Llave1986}). Hence the intersection of $\Pi_m$ with the center-stable manifolds in the neighborhood $U_m$, assuming $U_m$ is sufficiently small, gives a family of analytic curves $\set{\vartheta}$ in $\Pi_m$ (see \cite[Proof of Theorem 2.1-iii]{Yessen2011}). Those curves that intersect $\overline{\gamma^*}$ can be parameterized continuously (in the $C^{k\geq 1}$-topology) via $\overline{\gamma^*}\ni n\mapsto \vartheta(n)$ if and only if $n\in\vartheta(n)\cap\overline{\gamma^*}$. This allows us to apply \cite[Lemma 6.4]{Bedford1993} and conclude that $\vartheta(n)$ intersects $\overline{\gamma^*}$ transversally for all, except possibly finitely many, $n\in\overline{\gamma^*}$. By compactness, $\overline{\gamma^*}$ intersects the center-stable manifolds transversally at all, except possibly finitely many, points along $\overline{\gamma^*}$. Observe that, with $(\mf{p},\mf{q})\neq (1,0)$,
\begin{align*}
\frac{\partial I\circ \gamma}{\partial\lambda} = \frac{\mf{q}(1 - \mf{p}^2)}{4\mf{p}^2}\neq 0.
\end{align*}
It follows that $\overline{\gamma^*}$ intersects the invariant surfaces $\set{S_V}_{V > 0}$ transversally. Let $m\in\overline{\gamma^*}\cap S_V$ be a point of transversal intersection with the center-stable manifold. Application of \cite[Proof of Theorem 2.1-iii]{Yessen2011} shows that
\begin{align}\label{eq_local_dim}
\lhdim(\overline{\gamma^*}, m) = \frac{1}{2}\hdim(\Omega_V).
\end{align}
Since $V\mapsto\hdim(\Omega_V)$ is continuous (in fact, analytic---see \cite[Theorem 5.23]{Cantat2009}) and the points of tangential intersection, if such exist, are isolated, \eqref{eq_local_dim} holds for all points of intersection of $\overline{\gamma^*}$ with the center-stable manifolds. This proves the continuity statement. That the local Hausdorff dimension is nonconstant follows by the observation in \cite[Proof of Theorem 2.1-iii]{Yessen2011}; that it is constant in the diagonal and the off-diagonal cases follows from the observation that in these cases $I\circ\gamma(\lambda) > 0$ is $\lambda$-independent (see \cites{Damanik2009, Damanik2010}).

\subsubsection{Proof of ii-(a)}

Let $\lambda_0: \R^2\setminus\set{(1,0)}\rightarrow \R$ be such that $I\circ\gamma_{(\mf{p},\mf{q})}\circ\lambda_0(\mf{p},\mf{q}) = 0$. Define
\begin{align*}
\mathcal{C} = \set{(x,y,z): I(x,y,z) - V = 0\text{\hspace{2mm}and\hspace{2mm}}\abs{x},\abs{y},\abs{z}\leq 1}^\mathrm{c}.
\end{align*}
Then $\mathcal{C}$ is a smooth two-dimensional submanifold of $\R^3$ with four connected components (see, for example, \cites{Baake1999} and \cites{Roberts1994, Roberts1994b}), and the map $F: \R^2\setminus\set{(1,0)}\rightarrow\mathcal{C}$ defined as $F(\mf{p},\mf{q}) = \gamma_{(\mf{p},\mf{q})}\circ\lambda_0(\mf{p},\mf{q})$
%
%
is smooth. There exist four smooth curves in $\mathcal{C}$, whose union we denote by $\tau$, such that for all $x\in\mathcal{C}$, $\mathcal{O}_{f_0}^+(x)$ is bounded if and only if $x\in\tau$ (see \cites{Cantat2009, Damanik2009}). Let $\mf{N} = F^{-1}(\tau)$. Then $\mf{N}$ has zero Lebesgue measure, and for all $(\mf{p},\mf{q})\notin\mf{N}$, the intersection of the corresponding $\overline{\gamma^*}$ with the center-stable manifolds is away from $S_0$. Now using \eqref{eq_local_dim} together with the fact that 
\begin{align}\label{bounds_on_dim}
\text{for all\hspace{2mm}} V > 0,\text{\hspace{2mm}} 0 < \hdim(\Omega_V) < 2
\end{align}
(see \cites{Damanik2009c, Cantat2009}), we obtain \textit{ii-(a)}.

\subsubsection{Proof of ii-(b)}

Let $P = (1,1,1)$. One of the four curves mentioned above is a branch of the \textit{strong stable manifold at} $P$, which we denote by $W^\mathrm{ss}$; the tangent space $T_PW^\mathrm{ss}$ is spanned by the eigenvector of the differential of $f$ at $P$ corresponding to the smallest eigenvalue (see \cite[Section 4]{Damanik2009}). A simple computation, which we omit here, shows that $T_PW^\mathrm{ss}$ is transversal to the plane $\set{z = 1}$. Hence for all $\mf{p}\approx 1$, $W^\mathrm{ss}\cap \set{z = \frac{1 + \mf{p}^2}{2\mf{p}}}\neq \emptyset$. On the other hand, the first coordinate of $\gamma$ depends only on $\mf{q}$; hence, evidently from \eqref{gamma_defn}, for any $x\in\set{z = \frac{1 + \mf{p}^2}{2\mf{p}}}$ there exists $\mf{q}$ such that $x\in\gamma^*_{(\mf{p},\mf{q})}$. Thus, $\mf{N}\neq \emptyset$.

Let $(\mf{p},\mf{q})\in\mf{N}$, and $m\in\overline{\gamma^*_{(\mf{p},\mf{q})}}\cap S_0$. Then $\gamma^{-1}(\set{m})$ is one of the two extreme boundary points of the spectrum, and away from it, by \eqref{eq_local_dim} and \eqref{bounds_on_dim}, the local Hausdorff dimension is strictly between zero and one. On the other hand,
\begin{align}\label{eq_lim_dim}
\lim_{V\rightarrow 0^+}\hdim(\Omega_V) = 2
\end{align}
(see \cite[Theorem 1.1]{Damanik2010}). Hence $\hdim(\Sigma_{(\mf{p},\mf{q})}) = 1$.

\subsubsection{Proof of iii} This follows from \eqref{eq_lim_dim}, since $\overline{\gamma^*_{(\mf{p},\mf{q})}}$ depends continuously on $(\mf{p},\mf{q})$, and is close to $S_0$ whenever $(\mf{p},\mf{q})$ is close to $(1,0)$ (see equation \eqref{eq_fv_value}).

\subsubsection{Proof of iv} This follows, since $V_\mf{p}:= I\circ\gamma_{(\mf{p},0)}$ depends analytically on $\mf{p}$, and at the same time $\hdim(\Omega_{V_{\mf{p}}})$ depends analytically on $V_\mf{p}$ (see \cite[Theorem 5.23]{Cantat2009}); similarly with $(1,\mf{q})$. 

\subsection{Proof of theorem \ref{thm_main1}}\label{sec:proof_main1}

In what follows, for a regular curve $\alpha$ in $\R^n$, by $\alpha^*$ we denote the image of $\alpha$; the length of $\alpha$ is denoted by $\len[\alpha^*]$, and for any $a, b\in\alpha^*$, the distance along $\alpha^*$ between $a$ and $b$ is denoted by $\dist_{\alpha^*}(a,b)$ (i.e. the length of the arc along $\alpha^*$ connecting $a$ and $b$).

We also assume, unless stated otherwise, that $(\mf{p},\mf{q})\neq (1,0)$, and we always have $\mf{p}\neq 0$.

\begin{prop}\label{prop:trans}
The conclusion of Theorem \ref{thm_main1} holds for all $(\mf{p},\mf{q})$ such that $\gamma_{(\mf{p},\mf{q})}$ intersects the center-stable manifolds transversally.
\end{prop}
\begin{proof}[Proof of Proposition \ref{prop:trans}]
All intersections of $\gamma^*$ with the center-stable manifolds occur only on a compact line segment along $\gamma^*$; denote this segment by $\overline{\gamma^*}$. The Fricke-Vogt invariant along $\gamma$ takes values
\begin{align}\label{eq:fvi}
I\circ\gamma(E) = \frac{E\mf{q}(1-\mf{p}^2) + \mf{q}^2\mf{p}^2 + (\mf{p}^2 - 1)^2}{4\mf{p}^2}.
\end{align}
This gives
\begin{align}\label{eq:fvi-deriv}
\frac{\partial I\circ \gamma}{\partial E} = \frac{\mf{q}(1 - \mf{p}^2)}{4\mf{p}^2}\neq 0.
\end{align}
Hence $\gamma$ intersects the level surfaces $\set{S_V}_{V \geq 0}$ transversally. Notice that $\gamma$ lies in the plane $\Pi_\mf{p} := \set{z = \frac{1 + \mf{p}^2}{2\mf{p}}}$ (see \eqref{gamma_defn}). Let $\mathcal{T}$ be a neighborhood of $\gamma$. If $\mathcal{T}$ is sufficiently small, then, by transversality and \eqref{eq:fvi-deriv}, $\Pi_\mf{p}$ intersects the center-stable manifolds as well as the level surfaces transversally inside $\mathcal{T}$, and $\widetilde{\mathcal{T}}:=\mathcal{T}\cap\Pi_{\mf{p}}$ gives a neighborhood of $\gamma$ in $\Pi_\mf{p}$. The intersection of $\Pi_\mf{p}$ with the center-stable manifolds gives a family of smooth curves in $\widetilde{\mathcal{T}}$, which we denote by $\set{\vartheta}$. The intersection of $\Pi_\mf{p}$ with the invariant surfaces gives a family of smooth curves, $\set{\tau_V = \Pi_\mf{p}\cap S_V}_{V\geq 0}$, which smoothly foliate $\widetilde{\mathcal{T}}$.
\begin{lem}\label{lem:lip}
For every intersection point $m$ of $\overline{\gamma^*}$ with the center-stable manifolds, there exists $\epsilon_m, C_m > 0$ such that the following holds. If $m\in\tau_{V_m}^*$,  $V_m > 0$, then for every $n\in\tau_{V_m}^*$, $n\neq m$, with $\dist_{\tau_{V_m}^*}(n,m) < \epsilon_m$,
\begin{align}\label{eq:lem-lip}
\left(\frac{\dist_{\tau_{V_m}^*}(n,m)}{\dist_{\gamma^*}(n, \widetilde{n})}\right)^{\pm 1}\leq C_m,
\end{align}
where $\widetilde{n}$ is the intersection point of $\overline{\gamma^*}$ with the curve $\vartheta$ from $\set{\vartheta}$ going through $n$. 
\end{lem}
\begin{proof}[Proof of Lemma \ref{lem:lip}]
We begin with the following result, which will make matters easier later.
\begin{lem}\label{lem:curve-ratio}
Let $K_\eta(v)$ denote the cone around $v\in\R^n$ of angle $\eta$:
\begin{align*}
K_\eta(v) := \set{u\in\R^n: \measuredangle(u,v) < \eta}.
\end{align*}
For $\eta < \pi/4$, for any $\epsilon\in[0,\eta]$ there exists $M = M(\epsilon) \geq 1$ such that for any regular curve $\alpha: [0,1]\rightarrow \R^n$ satisfying $\alpha'(t)\in K_\epsilon(\alpha'(0))$ for all $t$, we have $\len[\alpha^*]/\norm{\alpha(0) - \alpha(1)}\leq M$. 
\end{lem}
\begin{proof}[Proof of Lemma \ref{lem:curve-ratio}]
Let $x_1,\dots,x_n$ be the axes of $\R^n$. We may assume that $\alpha(0),\alpha(1)\in x_1$. Hence $x_1\in K_\eta(\alpha'(0))$. By regularity, if $\alpha'_1(t) = 0$ implies that $\measuredangle(x_1,\alpha'(t)) = \pi/2$, contradicting the hypothesis. Hence $\alpha'_1(t)\neq 0$ for any $t$, and we may parameterize $\alpha$ along $x_1: \alpha(t) = (t, \alpha_2(t),\dots,\alpha_n(t))$ with $t\in[\alpha(0),\alpha(1)]\subset x_1$. We have $\abs{\alpha'_j(t)} = \tan\theta$, where $\theta$ is the angle between $x_1$ and the projection of $\alpha'(t)$ onto the $(x_1,x_j)$-plane. Since $\measuredangle(\alpha'(t),x_1) < 2\epsilon$, we have $\theta < 2\epsilon$, hence $\abs{\alpha'_j(t)}<\tan2\epsilon$. Now,
\begin{align*}
\len[\alpha^*] &= \int_{\alpha(0)}^{\alpha(1)}\norm{\alpha'(t)}dt\\
&\leq \int_{\alpha(0)}^{\alpha(1)}\sum_j\abs{\alpha'_j(t)}dt\leq [\alpha(1) - \alpha(0)][1 + (n-1)\tan2\epsilon].
\end{align*}
The result follows with $M = [1 + (n-1)\tan2\epsilon]$. 
\end{proof}
Parameterize the curves $\set{\vartheta}$ by $V$ with $\vartheta(V) = \vartheta\cap\tau_V^*$ (which is made possible by transversality of intersection of the center-stable manifolds with the level surfaces $\set{S_V}_{V > 0}$ --- see Proposition 3.9 and proof of Theorem 2.1-iii in \cite{Yessen2011}). Parameterize the subfamily of $\set{\vartheta}$ of curves that intersect $\tau_{V_m}^*$ inside $\widetilde{\mathcal{T}}$ by $n\mapsto \vartheta_n$, where $\set{n} = \vartheta_n^*\cap\tau_{V_m}^*$. Define two constant cone fields $K_\eta^{\mathrm{ver}}$ and $K_\eta^\mathrm{hor}$ on $\Pi_\mf{p}$, transversal to each other, where $0 < \eta < \pi/4$ is such that $\vartheta_m$ is tangent to $K_\eta^{\mathrm{ver}}$ at $m$, $\tau_{V_m}$ is tangent to $K_\eta^{\mathrm{hor}}$ at $m$, and $\gamma^*$ is transversal to both cones. Let $\delta > 0$ such that $V_m - \delta > 0$ and set $\widetilde{\vartheta_n^*} = \vartheta_n[V_m - \delta, V_m + \delta]$. Now, taking $\delta$ sufficiently small, we have $\widetilde{\vartheta_m^*}$ tangent everywhere to $K_\eta^{\mathrm{ver}}$. Similarly, let $\widetilde{\tau_{V_m}^*}$ be a compact arc along $\tau_{V_m}^*$ containing $m$ in its interior; assuming the arc is sufficiently short, we have $\widetilde{\tau_{V_m}^*}$ tangent everywhere to $K_\eta^{\mathrm{hor}}$. The curves $\vartheta_n$ depend continuously on $n\in\tau_{V_m}^*$ in the $C^1$-topology (see \cite[Proposition 3.9]{Yessen2011}), hence if $\epsilon_m$ is sufficiently small, then for all $n\in\widetilde{\tau_{V_m}^*}$ with $\dist_{\tau_{V_m}^*}(n,m) < \epsilon_m$, $\widetilde{\vartheta_n^*}$ intersects $\gamma^*$ in one point and is everywhere tangent to $K_\eta^{\mathrm{ver}}$. Let $L_n^{\mathrm{ver}}$ denote the line segment connecting points $n$ and $\tilde{n}$ -- the point of intersection of $\widetilde{\vartheta_n^*}$ with $\gamma^*$, and $L_n^{\mathrm{hor}}$ the line segment connecting $m$ and $n$. If $n\neq m$ and the distance between $n$ and $m$ is not greater than $\epsilon_m$, by the mean value theorem $L_n^{\mathrm{ver},\mathrm{hor}}$ is tangent to, respectively, $K_\eta^{\mathrm{ver},\mathrm{hor}}$. It follows that $L_n^{\mathrm{ver},\mathrm{hor}}$ is transversal to $\gamma^*$ uniformly in $n$, and hence there exists $\widetilde{C}_m > 0$, such that for all $n\neq m$ whose distance from $m$ is not greater than $\epsilon_m$,
\begin{align}\label{eq:lem-lip2}
\left(\frac{\len(L_n^\mathrm{hor})}{\dist_{\gamma^*}(n, \widetilde{n})}\right)^{\pm 1}\leq \widetilde{C}_m.
\end{align}
Now application of Lemma \ref{lem:curve-ratio} allows to replace $\len(L_n^\mathrm{hor})$ in inequality \eqref{eq:lem-lip2} with the distance between $m$ and $n$, $\dist_{\tau_{V_m}^*}(m, n)$, to obtain \eqref{eq:lem-lip} with
$C_m = M\widetilde{C}_m$, where $M$ is as in Lemma \ref{lem:curve-ratio}. 
\end{proof}
\begin{rem}\label{rem:lip}
The families $\set{\vartheta}$ and $\set{\tau_V}_{V > 0}$ can be parameterized by $n\mapsto \vartheta_n$ and $n\mapsto \tau_n\in\set{\tau_V}$ where $\set{n} = \overline{\gamma^*}\cap \vartheta_n$ and $\set{n} = \overline{\gamma^*}\cap \tau_n$, respectively. In this parameterization, $\vartheta_n$ and $\tau_n$ depend continuously on $n$ in the $C^1$-topology. Hence, by compactness of $\overline{\gamma^*}$, in Lemma \ref{lem:lip} one can choose $\epsilon, C$ independent of $m$.
\end{rem}
Recall that a morphism $H: (M_1, d_1)\rightarrow (M_2,d_2)$ of metric spaces is called H\"older continuous, or simply H\"older, if there exist a \textit{constant} $K > 0$ and \textit{exponent} $\alpha\in (0,1]$ such that for all $x,y\in M_1$, $d_2(H(x),H(y))\leq Kd_1(x, y)^\alpha$.

Denote by $\Gamma$ the intersection of $\overline{\gamma^*}$ with the center-stable manifolds. Denote by $T_V$ the intersection of $\tau_V^*$ with the curves $\set{\vartheta}$.  Let $H_{V_1,V_2}: T_{V_1}\rightarrow T_{V_2}$ be the holonomy map defined by projecting points along the curves $\set{\vartheta}$. Note that $H_{V_1,V_2}$ is a homeomorphism.
\begin{lem}\label{lem:holder}
Let $m\in\Gamma$ with $m\in\tau_{V_m}^*$, $V_m > 0$. Let $h$ be the holonomy map defined in a neighborhood (along $\tau_{V_m}^*$) of $m$ by projecting points from $T_{V_m}$ to $\Gamma$ along the curves $\set{\vartheta}$. Then for every $\alpha\in(0,1)$ there exists $\epsilon_\alpha > 0$ such that the following holds. If $\tau^*$ is a compact arc along $\tau_{V_m}^*$ containing $m$ in its interior and $\len[\tau^*] < \epsilon_\alpha$, then $h|_{T_{V_m}\cap\tau^*}$ and its inverse are H\"older, both with exponent $\alpha$. 
\end{lem}
\begin{proof}[Proof of Lemma \ref{lem:holder}]
Let $C, \epsilon > 0$ be as in Remark \ref{rem:lip}. Let $\epsilon_\alpha > 0$ be so small, that for all $n,n'\in T_{V_m}\cap\tau^*$, $n\neq n'$, the following holds. If $h(n')\in T_V$,  then
\begin{align*}
\dist_{\tau_V^*}(h(n'),H_{V_m, V}(n)) = \dist_{\tau_V^*}(H_{V_m, V}(n'), H_{V_m, V}(n)) < \epsilon.
\end{align*}
Then by Lemma \ref{lem:lip}, we get
\begin{align}\label{eq:holder1}
&\left(\frac{\dist_{\tau_V^*}(h(n'),H_{V_m, V}(n))}{\dist_{\gamma^*}(h(n'),h(n))}\right)^{\pm 1}\\\notag
&= \left(\frac{\dist_{\tau_V^*}(H_{V_m, V}(n'), H_{V_m,V}(n))}{\dist_{\gamma^*}(h(n'), h(n))}\right)^{\pm 1}\leq C.
\end{align}
By \cite[Lemma 4.21]{Yessen2011}, there exist $\delta, K> 0$ such that $V_m - \delta > 0$ and for all $V\in[V_m - \delta, V_m + \delta]$, $H_{V_m, V}$ and its inverse are both H\"older with constant $K$ and exponent $\alpha$. By taking $\epsilon_\alpha$ smaller as necessary, we can ensure that for all $n\in T_{V_m}\cap \tau^*$, if $h(n)\in T_V$, then $V\in[V_m - \delta, V_m + \delta]$. Combining this with \eqref{eq:holder1} completes the proof. 
\end{proof}
Denote by $\underline{\bdim}$ and $\overline{\bdim}$ the lower and upper box-counting dimensions, respectively. Note that $T_V$ is a dynamically defined Cantor set (see \cite[Chapter 4]{Palis1993} for definitions). As a consequence, for every $n\in T_V$,  $\lbdim(n, T_V)$ exists and
\begin{align}\label{eq:prop1}
\bdim(T_V) = \lbdim(n, T_V) = \lhdim(n, T_V) = \hdim(T_V).
\end{align}
As a consequence of \eqref{eq:prop1} and Lemma \ref{lem:holder} we obtain the following. For every $m\in\Gamma\cap T_V$ and $\alpha\in(0,1)$ there exists $\epsilon_{m,\alpha}>0$ such that for any compact arc $\beta^*$ along $\gamma^*$ containing $m$ in its interior and $\len[\beta^*] < \epsilon_{m,\alpha}$, we have
\begin{align}\label{eq:prop2}
\alpha\hdim(T_V)\leq \hdim(\Gamma\cap\beta^*)&\leq \underline{\bdim}(\Gamma\cap\beta^*)\leq \overline{\bdim}(\Gamma\cap\beta^*)\\
&\leq \frac{1}{\alpha}\overline{\bdim}(T_V) = \frac{1}{\alpha}\hdim(T_V)\notag,
\end{align}
where $V$ is such that $x\in T_V$.

Now let $\beta^*$ be any compact arc along $\gamma^*$ containing $m$ in its interior. Let $\alpha\in(0,1)$. Pick a sequence of points $m_1,\dots,m_l$ in $\beta^*\cap \Gamma$, with $m_j\in T_{V_j}$, and partition $\beta^*$ into sub-arcs $\beta_1^*,\dots,\beta_l^*$ such that $m_j\in\beta_j^*$ and, by \eqref{eq:prop2},
\begin{align}\label{eq:prop3}
\alpha\hdim(T_{V_j})\leq \underline{\bdim}(\Gamma\cap\beta_j^*)\leq \overline{\bdim}(\Gamma\cap\beta_j^*)\leq \frac{1}{\alpha}\hdim(T_{V_j}).
\end{align}
Say $\max_{1\leq j\leq l}\set{\overline{\bdim}(\Gamma\cap\beta_j^*)} = \overline{\bdim}(\Gamma\cap\beta_{j_0}^*)$. Then via basic properties of lower and upper box-counting dimensions (see, for example, \cite[Theorem 6.2]{Pesin1997}), we have
\begin{align}\label{eq:prop4}
\overline{\bdim}(\Gamma\cap\beta^*) - \underline{\bdim}(\Gamma\cap\beta^*)& \leq \overline{\bdim}(\Gamma\cap\beta_{j_0}^*) - \max_{1\leq j \leq l}\set{\underline{\bdim}(\Gamma\cap\beta_{j}^*)}\\
&\leq \overline{\bdim}(\Gamma\cap\beta_{j_0}^*) - \underline{\bdim}(\Gamma\cap\beta_{j_0}^*)\notag.
\end{align}
In view of \eqref{eq:prop3}, the right side of \eqref{eq:prop4} can be made arbitrarily small by taking $\alpha$ sufficiently close to one. Hence $\overline{\bdim}(\Gamma\cap\beta^*) = \underline{\bdim}(\Gamma\cap\beta^*)$, and so $\bdim(\Gamma\cap\beta^*)$ exists. This proves the first assertion of the proposition. That local Hausdorff and box-counting dimensions coincide follows from \eqref{eq:prop2}. Hence, by continuity, both local box-counting and local Hausdorff dimensions are maximized simultaneously at some point in the spectrum. This shows equality of global Hausdorff and box-counting dimensions.
\end{proof}
\begin{rem} In the proof above, we assumed that the intersections occur away from the surface $S_0$ (i.e. the assumption in Lemmas \ref{lem:curve-ratio} and \ref{lem:holder} that $V_m > 0$). This need not always be the case; however, if an intersection does occur on $S_0$, then it occurs in a unique point that corresponds to one of the extreme boundaries of the spectrum, and at this point the local Hausdorff dimension is maximal (equals one).
\end{rem}

To complete the proof of Theorem \ref{thm_main1} it is enough to prove, by Proposition \ref{prop:trans}, that for the values $(\mf{p},\mf{q})$ given in the statement of the theorem, the corresponding line of initial conditions intersects the center-stable manifolds transversally. We do this next.

\begin{prop}\label{prop:trans-1}
For all $(\mf{p},\mf{q})\approx (1,0)$ and not equal to $(1,0)$, $\gamma_{(\mf{p},\mf{q})}$ intersects the center-stable manifolds transversally.
\end{prop}
\begin{proof}[Proof of Proposition \ref{prop:trans-1}]
As we recalled above, $S_{V > 0}$ is a two-dimensional non-compact connected analytic submanifold of $\R^3$; $S_0$, however, is smooth everywhere except for four conic singularities: $P_1 = (1,1,1)$, $P_2 = (-1,-1,1)$, $P_3 = (1,-1,-1)$ and $P_4 - (-1,1,-1)$. Let 
\begin{align*}
\mathbb{S} = \set{(x,y,z)\in S_0: \abs{x},\abs{y},\abs{z}\leq 1}.
\end{align*}
Then $\mathbb{S}$ is homeomorphic to the two-sphere and $f(\mathbb{S}) = \mathbb{S}$. Moreover, $f|_{\mathbb{S}}$ is a factor of the hyperbolic automorphism $\mathcal{A} = \bigl(\begin{smallmatrix}1 & 1\\ 1 & 0\end{smallmatrix}\bigr)$ on the two-torus $\mathbb{T}^2$, given by
\begin{align}\label{eq:ptrans-1}
F: (\theta,\phi)\mapsto(\cos2\pi(\theta + \phi), \cos2\phi\theta, \cos2\pi\phi).
\end{align}
Let $U_i$ be a small neighborhood of $P_i$. Set $U = \bigcup_i U_i$. For all $V > 0$ sufficiently small, $S_0\setminus U$ and $S_V\setminus U$ are smooth manifolds (with boundary) consisting of five connected components, one of which is compact; denote the compact component by $\mathbb{S}_{V, U}$. The unstable cone family for $\mathcal{A}$ on $\mathbb{T}^2$ can be carried to $\mathbb{S}_{0, U}$ via $DF$ and extended to all $\mathbb{S}_{V, U}$, for $V$ sufficiently small (see \cite{Damanik2009} for details). Denote this field by $\mathcal{K}_V$. With $V_0$ sufficiently small, define the following cone field on $\bigcup_{0 < V < V_0}\mathbb{S_{V,U}}$:
\begin{align}\label{eq:ptrans-2}
K_V^\eta(x) = \set{(\mathbf{u},\mathbf{v})\in T_x\mathbb{S}_{V, U}\oplus (T_x\mathbb{S}_{V, U})^\perp: \mathbf{u}\in\mathcal{K}_V(x)\text{, }\norm{\mathbf{v}}\leq\eta\sqrt{V}\norm{\mathbf{u}}}.
\end{align}
From \cite{Yessen2011} we have the following
\begin{lem}\label{lem:3dcones}
There exists $\eta > 0$ such that for all $V > 0$ sufficiently small, the cones $\set{K_V^\eta(x)}_{x\in{S}_{V, U}}$ are transversal to the center-stable manifolds.
\end{lem}
Intersections of $\gamma$ with the center-stable manifolds occur on a compact segment along $\gamma^*$, which we denote by $\overline{\gamma^*}$, and which belongs to $\bigcup_{V > 0}S_V$. Set, for convenience, $V(E) = I\circ\gamma(E)$. If $E_0$ denotes the unique value for which $V(E_0) = 0$, then away from $E_0$, from \eqref{eq:fvi} and \eqref{eq:fvi-deriv} we obtain
\begin{align}\label{eq:ptrans-3}
\frac{\partial V(E)}{\partial E}\cdot V(E)^{-1} &= \frac{\mf{q}(1 - \mf{p}^2)}{E\mf{q}(1 - \mf{p}^2) + \mf{q}^2\mf{p}^2 + (\mf{p}^2 - 1)^2} = \frac{1}{E - E_0}\\
&\implies \frac{\partial V(E)}{\partial E} = \frac{1}{E-E_0}V(E).\notag
\end{align}
Notice that $\gamma_{(1,0)}$ passes through $P_1$ and $P_2$, hence application of \cite[Proposition 3.1-(2)]{Yessen2011} shows that for all $(\mf{p},\mf{q})$ sufficiently close to $(1,0)$, intersections of $\gamma$ with the center-stable manifolds occur along $\overline{\gamma^*_{(\mf{p},\mf{q})}}$ that lies entirely inside $U\cup(\cup_{V > 0}\mathbb{S}_{V, U})$. On the other hand, intersection of $\gamma^*$ with $S_0$ occurs inside $U_1\cup U_2$, hence outside of $U$, $\abs{E - E_0}$ is bounded uniformly away from zero. Combining this with the fact that outside of $U$, $\nabla I(x,y,z)$ is bounded uniformly away from zero, using \eqref{eq:ptrans-3} we obtain that for all $(\mf{p},\mf{q})$ sufficiently close to $(1,0)$, $\overline{\gamma_{(\mf{p},\mf{q})}^*}$ is tangent to the cones $K_V^\eta$, with $\eta$ as in Lemma \ref{lem:3dcones}, and hence transversal to the center-stable manifolds (see proof of Corollary 4.12 in \cite{Yessen2011} for details). Therefore, we only need to investigate the situation in the vicinity of $\gamma^*\cap S_0$.
%
%
%
%
%
\begin{figure}[t]
\centering
 \subfigure{
\includegraphics[scale=.25]{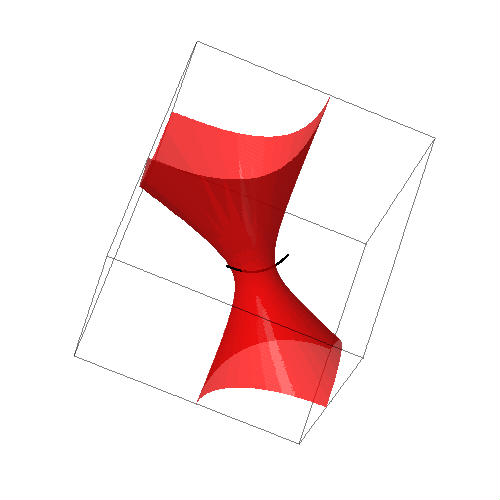}}
 \subfigure{
\includegraphics[scale=.25]{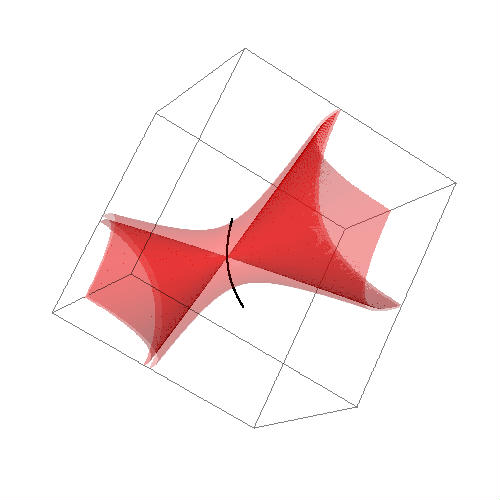}}
\caption{$\mathrm{Per}_2$ in a neighborhood of $P_1$.}
\label{part3_fig1}
\end{figure}

Let us first assume that $\gamma(E_0)\in U_1$. The set of period-two periodic points for $f$ passes through $P_1$ and forms a smooth curve in its vicinity (see Figure \ref{part3_fig1}):
\begin{align}\label{eq:ptrans-5}
\mathrm{Per}_2(f)=\set{(x,y,z): x\in\left(-\infty,1/2\right)\cup\left(1/2,\infty\right),\text{\hspace{2mm}}y=\frac{x}{2x - 1},\text{\hspace{2mm}}z = x}.
\end{align}
This curve is normally hyperbolic, and the stable manifold to this curve, which we denote by $W^\mathrm{cs}(P_1)$, is tangent to $S_0$ along the strong-stable manifold to $P_1$, denoted by $W^\mathrm{ss}(P_1)$ (see \cite{Damanik2009}). Let $O(P_1)$ be a small neighborhood of $P_1$ in $\R^3$ and define
\begin{align}\label{eq:local-mans}
W_\loc^\mathrm{cs}(P_1)&=\set{x\in\R^3: f^n(x)\in O(P_1)\text{ for all }n\in\N};\\
W_\loc^\mathrm{ss}(P_1)&=\set{x\in\ W_\loc^\mathrm{cs}(P_1): f^n(x)\rightarrow P_1\text{ as }n\rightarrow\infty}\notag.
\end{align}
The manifolds $W_\loc^\mathrm{cs}(P_1)$ and $W_\loc^\mathrm{ss}(P_1)$ are neighborhoods of $P_1$ in $W_\loc^\mathrm{cs}(P_1)$ and $W_\loc^\mathrm{cs}(P_1)$, respectively, contained in $O(P_1)$. The manifolds $W^\mathrm{cs}(P_1)$ and $W^\mathrm{ss}(P_1)$ are injectively immersed two- and one-dimensional submanifolds of $\R^3$, respectively. The manifold $W^\mathrm{ss}(P_1)$ consists of two smooth branches, one injectively immersed in $\mathbb{S}\setminus\set{P_1,\dots,P_4}$, the other in the cone of $S_0$ attached to $P_1$ (see Figure \ref{part2_fig1}), and these two branches connect smoothly at $P_1$.
\begin{lem}\label{lem:prop-help2}
For all $(\mf{p}, \mf{q})$ sufficiently close to $(1,0)$, $\gamma_{(\mf{p},\mf{q})}$ intersects $W_\loc^\mathrm{cs}(P_1)$ transversally in a unique point, call it $p$. The arc along $\gamma_{(\mf{p},\mf{q})}^*$ connecting $p$ and $\gamma_{(\mf{p},\mf{q})}(E_0)$ does not intersect the center-stable manifolds other than at $p$, where $E_0$ is the unique point such that $\gamma_{(\mf{p},\mf{q})}(E_0)\in S_0$.
\end{lem}
\begin{proof}[Proof of Lemma \ref{lem:prop-help2}]
The tangent space to $W^\mathrm{ss}(P_1)$ at $P_1$ is spanned by the eigenvector of $Df$ corresponding to the largest eigenvalue. After a simple computation, we get that 
\begin{align*}
T_{P_1}W^\mathrm{ss}(P_1)\oplus 
T_{P_1}\mathrm{Per}_2(f)\oplus 
T_{P_1}\gamma_{(1,0)}^* = \R^3.
\end{align*}
Hence $\gamma^*{(1,0)}$ intersects $W_\loc^\mathrm{cs}(P_1)$ transversally at the unique point $P_1$. Since $W_\loc^\mathrm{cs}(P_1)$ is a two-dimensional disc embedded in $\R^3$, all sufficiently small $C^1$ perturbations of $\gamma^*_{(1,0)}$ intersect $W_\loc^\mathrm{cs}(P_1)$ transversally in a unique point; this is true in particular for all $\gamma^*_{(\mf{p},\mf{q})}$ with $(\mf{p},\mf{q})\approx (1,0)$. 

Let $\mathcal{C}_{P_1}$ denote the cone of $S_0$ attached to $P_1$. If the arc connecting $p$ and $\gamma(E_0)$ intersects center-stable manifolds at points other than $p$, then the intersection of these center-stable manifolds with $\mathcal{C}_{P_1}$ will form a lamination of a neighborhood of $P_1$ in $\mathcal{C}_{P_1}$ consisting of uncountably many disjoint one-dimensional embedded submanifolds of $\mathcal{C}_{P_1}$, each point of which has bounded forward semi-orbit under $f$. On the other hand, a point in $\mathcal{C}_{P_1}$ has bounded forward semi-orbit if and only if it lies in $\widetilde{W}^\mathrm{ss}(P_1)$, the branch of $W^\mathrm{ss}(P_1)$ lying in $\mathcal{C}_{P_1}$ (this follows from general principles); hence this lamination must consist of pieces of $\widetilde{W}^\mathrm{ss}(P_1)$. Let $\widetilde{W}_\loc^\mathrm{ss}(P_1)$ denote the branch of $W_\loc^\mathrm{ss}(P_1)$ lying on $\mathcal{C}_{P_1}$. Then $\widetilde{W}^\mathrm{ss}(P_1) = \bigcup_{n\in\N}f^{-n}(\widetilde{W}_\loc^\mathrm{ss}(P_1))$. On the other hand, since the points of $S_0$ whose full orbit is bounded belong to $\mathbb{S}$, every point of $\widetilde{W}_\loc^\mathrm{ss}(P_1)$, not including $P_1$, must diverge under iterations of $f^{-1}$. Now, $f^{-1}(x,y,z) = (y,z,2yz-x) = \sigma\circ f\circ\sigma$, where $\sigma: (x,y,z)\mapsto (z,y,x)$ (see \cite{Baake1997} for more details on reversing symmetries of trace maps). Hence the results of \cite{Roberts1996} apply: unbounded backward semi-orbits under $f$ escape to infinity. It follows that pieces of $\widetilde{W}^\mathrm{ss}(P_1)$ cannot form the aforementioned lamination.
\end{proof}
\begin{prop}\label{prop:clen}
If $U_1$ is taken sufficiently small, then there exist $N_0\in\N$ and $C > 0$ such that the following holds. If $E$ is such that $\gamma(E)$ does not lie on the arc connecting $\gamma(E_0)$ and $p$ (with $p$ as in the previous lemma), and the arc connecting $\gamma(E)$ and $p$, which we denote by $\beta$, lies entirely in $U_1$, and if $k\in\N$ is the smallest number such that $f^k(\beta)\cap U_1^c\neq \emptyset$, then if $k \geq N_0$, we have $\norm{Df^k(\gamma'(E))}\abs{E-E_0}\geq C\norm{\gamma'(E)}$.
\end{prop}
\begin{proof}[Proof of Proposition \ref{prop:clen}]
Assuming $O(P_1)$ is taken sufficiently small, take a diffeomorphism $\Phi: O(P_1)\rightarrow \R^3$ such that
\begin{itemize}
\item $\Phi(P_1) = (0,0,0)$;
\item $\Phi(\mathrm{Per}_2(f))$ is part of the line $\set{x = 0,\text{\hspace{1mm}}z = 0}$;
\item $\Phi(W_\loc^\mathrm{cs}(P_1))$ is part of the plane $\set{z = 0}$.
\end{itemize}
Assume also that $\overline{U_1}\subset O(P_1)$. 
\begin{lem}\label{lem:prop-help1}
There exist $\lambda > 1$, $C^* > 0$, $C^{**} > 0$, and for every $\eta > 0$ there exist $C_1 > 0$ and $N_0\in\N$ such that the following holds. Define
\begin{align}\label{eq:large-cones}
\mathcal{K}^\eta = \set{(x,y,z) = (u,v)\in\R^2\oplus\R: \norm{v}\leq\eta\norm{u}},
\end{align}
and let $\widetilde{f} = \Phi\circ f\circ \Phi^{-1}$.
\begin{enumerate}[i.]
\item For all $x\in\Phi(U_1)$, if $k\in\N$ is such that $\widetilde{f}^{k-1}(x)\in \Phi(U_1)$, $\widetilde{f}^k(x)\notin \Phi(U_1)$ and $k\geq N_0$, then for any $v\in T_x\R^3$ with $v\in\mathcal{K}^\eta$, $\norm{D\widetilde{f}^k(v)}\geq C_1\lambda^k\norm{v}$.
\item If $x_z$ denotes the $z$-component of $x$, then $C^*\lambda^{-k}\leq x_z \leq C^{**}\lambda^{-k}$.
\end{enumerate}
\end{lem}
\begin{proof}[Proof of Lemma \ref{lem:prop-help1}]
For the first assertion, one needs to notice that the cones in \eqref{eq:large-cones}, unlike those defined in \cite[Proposition 3.15]{Damanik2010}, have fixed width. This allows us to replace the inequality $\norm{D\widetilde{f}^k(v)}\geq C_1\lambda^{k/2}$ in \cite[Proposition 3.15]{Damanik2010} with $\norm{D\widetilde{f}^k(v)}\geq C_1\lambda^{k}$.

The second assertion is a restatement of \cite[Proposition 3.14]{Damanik2010}.
\end{proof}
Let $m$ be a point in $\Phi(\beta)$ such that $\widetilde{f}^k(m)\notin \Phi(U_1)$, $\widetilde{f}^{k-1}(m)\in\Phi(U_1)$. Let $\widetilde{\beta}$ denote the arc along $\Phi(\beta)$ connecting $m$ and $\Phi(p)$. We have 
\begin{align*}
0 < m_z\leq\len[\Phi(\widetilde{\beta})]\leq\len[\Phi(\beta)].
\end{align*}
Let $v\in T_{m}\R^3$ with $v\in \mathcal{K}^\eta$. Application of Lemma \ref{lem:prop-help1} gives 
\begin{align*}
\norm{D\widetilde{f}^k(v)}m_z\geq C_1C^*\norm{v}.
\end{align*}
Hence we have $\norm{D\widetilde{f}^k(v)}\len[\Phi(\beta)]\geq C_1C^*\norm{v}$. On the other hand, since, by Lemma \ref{lem:prop-help2}, $\gamma^*_{(\mf{p},\mf{q})}$ is uniformly transversal to $W_\loc^\mathrm{cs}(P_1)$ for all $(\mf{p},\mf{q})$ sufficiently close to $(1,0)$, for $U_1$ sufficiently small there exists $\eta > 0$ such that for all $(\mf{p},\mf{q})\approx (1,0)$, $\Phi(\gamma^*_{(\mf{p},\mf{q})}\cap U_1)$ is tangent to $\mathcal{K}^\eta$. This completes the proof.
\end{proof}
\begin{rem}\label{rem:better-clen}
The bound $C\norm{\gamma'(E)}$ in the conclusion of Proposition \ref{prop:clen} can be replaced with a constant, say $\widetilde{C}$, since for all $(\mf{p},\mf{q})$ with $\mf{p}$ uniformly away from zero, $\norm{\gamma'(E)}$ is uniformly away from infinity (see \eqref{gamma_defn}).
\end{rem}
Let $U_i^*$ be a neighborhood of $P_i$ such that for all $m\in U_1^*$, if $f^k(m)\notin U_1$, then $k > N_0$, with $N_0$ as in Proposition \ref{prop:clen}. For all $(\mf{p},\mf{q})$ sufficiently close to $(1,0)$, $\overline{\gamma^*_{(\mf{p},\mf{q})}}$, the compact line segment along $\gamma^*_{(\mf{p},\mf{q})}$ on which intersections with center-stable manifolds occur, has its endpoints inside $U_1^*\cup U_2^*$. If $E\in\R$ is such that $\gamma_{(\mf{p},\mf{q})}(E)\in U_1^*$ is a point of intersection with a center-stable manifold, and if for all $k$, $f^k(\gamma_{(\mf{p},\mf{q})}(E))\in U_1$, then $\gamma_{(\mf{p},\mf{q})}(E)\in W_\loc^\mathrm{cs}(P_1)$, hence $\gamma_{(\mf{p},\mf{q})}(E)$ coincides with $p$ of Lemma \ref{lem:prop-help2}, and this intersection is transversal. Otherwise, say $k\in\N$ is such that $f^k(\gamma_{(\mf{p},\mf{q})}(E))\notin U_1$ and $f^{k-1}(\gamma_{(\mf{p},\mf{q})}(E))\in U_1$. We have
\begin{align*}
\norm{\proj{(T_{\gamma_{(\mf{p},\mf{q})}(E)}S_{V(E)})^\perp}{\gamma'_{(\mf{p},\mf{q})}(E)}} = \frac{\partial V(E)}{\partial E} \nabla I(\gamma_{(\mf{p},\mf{q})}(E))^{-1}
\end{align*}
(recall: $V(E) = I\circ\gamma(E)$). On the other hand, by \cite[Lemma 4.9]{Yessen2011} we have
\begin{align*}
&\norm{\proj{(T_{f^k(\gamma_{(\mf{p},\mf{q})}(E))}S_{V(E)})^\perp}{\gamma'_{(\mf{p},\mf{q})}(E)}}\\
&= \frac{\nabla I(\gamma_{(\mf{p},\mf{q})}(E))}{\nabla I(f^k(\gamma_{(\mf{p},\mf{q})}(E)))}\norm{\proj{(T_{\gamma_{(\mf{p},\mf{q})}(E)}S_{V(E)})^\perp}{\gamma'_{(\mf{p},\mf{q})}(E)}}.
\end{align*}
Hence we obtain
\begin{align*}
\norm{\proj{(T_{f^k(\gamma_{(\mf{p},\mf{q})}(E))}S_{V(E)})^\perp}{\gamma'_{(\mf{p},\mf{q})}(E)}} &= \frac{\partial V(E)}{\partial E}\nabla I(f^k(\gamma_{(\mf{p},\mf{q})}(E)))^{-1}\\
&\leq \frac{1}{D}\frac{\partial V(E)}{\partial E},
\end{align*}
where $D > 0$ is the lower bound of the gradient of $I$ restricted to $\mathbb{S}_{V, U}$. Therefore,
\begin{align*}
&\norm{\proj{(T_{f^k(\gamma_{(\mf{p},\mf{q})}(E))}S_{V(E)})^\perp}{\gamma'_{(\mf{p},\mf{q})}(E)}}\left(\norm{Df^k(\gamma'_{(\mf{p},\mf{q})}(E))}V(E)\right)^{-1}\\
&\leq \frac{1}{D}\frac{\partial V(E)}{\partial E}\left(\norm{Df^k(\gamma'_{(\mf{p},\mf{q})}(E))}V(E)\right)^{-1}\\
&= \frac{1}{D\norm{Df^k(\gamma'_{(\mf{p},\mf{q})}(E))}\abs{E - E_0}}\leq \frac{1}{D\widetilde{C}}
\end{align*}
(the last equality follows from \eqref{eq:ptrans-3}), where $\widetilde{C}$ is as in Remark \ref{rem:better-clen}. Finally, with \eqref{eq:ptrans-3} in mind, we obtain
\begin{align*}
\norm{\proj{(T_{f^k(\gamma_{(\mf{p},\mf{q})}(E))}S_{V(E)})^\perp}{\gamma'_{(\mf{p},\mf{q})}(E)}}\norm{Df^k(\gamma'_{(\mf{p},\mf{q})}(E))}^{-1}\leq \frac{1}{D\widetilde{C}}V(E).
\end{align*}
Hence if $V(E)$ is small (i.e., for all $(\mf{p},\mf{q})$ sufficiently close to $(1,0)$), $Df^k(\gamma'_{(\mf{p},\mf{q})})$ is tangent to the cone $K_{V(E)}^\eta$, with $\eta$ as in Lemma \ref{lem:3dcones}. By invariance of the center-stable manifolds under $f$ and Lemma \ref{lem:3dcones} it follows that the intersection of $\gamma_{(\mf{p},\mf{q})}^*$ with center-stable manifold at $\gamma_{(\mf{p},\mf{q})}(E)$ is transversal. Thus, for all $(\mf{p},\mf{q})$ sufficiently close to $(1,0)$, if  $\gamma_{(\mf{p},\mf{q})}(E_0)\in U_1$, then $\gamma_{(\mf{p},\mf{q})}^*$ intersects the center-stable manifolds transversally inside $U_1^*$. An argument similar to the one above, with $U^* = \bigcup_i U_i^*$ in place of $U$, shows that outside of $U^*$ the intersections are also transversal. It remains to investigate the case when $\gamma_{(\mf{p},\mf{q})}(E_0)\in U_2$.

In case $\gamma_{(\mf{p},\mf{q})}(E_0)\in U_2$, we can reduce everything to the previous case as follows. Replace, without loss of generality, $f$ with $f^3$. Let $\sigma: (x,y,z)\mapsto(-x,-y,z)$. Notice that $\sigma$ is simply rotation in the $xy$-plane around the origin by $\pi$, $\sigma$ preserves $S_V$ for all $V$, $f^3 = \sigma^{-1}\circ f^3\circ\sigma = \sigma\circ f^3\circ \sigma$, and $\sigma$ maps $P_1$ to $P_2$. Essentially, all of this guarantees that one can rotate the line $\gamma^*$ by $\pi$ in the $xy$-plane while keeping all other geometric objects invariant (i.e. the level surfaces $S_V$ as well as center-stable manifolds), thus reducing everything to the previous case.

The proof of Proposition \ref{prop:trans-1} is complete.
\end{proof}
\begin{prop}\label{prop:trans-2}
There exists $\Delta > 0$ such that for all $\mf{p}$ satisfying $\abs{\mf{p} - 1} > \Delta$ and all $\mf{q}$ satisfying $\abs{\mf{q}} > \Delta$, there exist $\delta_{\mf{p}}, \delta_{\mf{q}}>0$, such that for all $\alpha$ in the interval $(1-\delta_{\mf{p}}, 1 + \delta_{\mf{p}})$ and $\beta\in(-\delta_{\mf{q}},\delta_{\mf{q}})$, $\gamma_{(\alpha,\mf{q})}^*$ and $\gamma_{(\mf{p},\beta)}^*$ intersect the center-stable manifolds transversally.
\end{prop}
\begin{proof}[Proof of Proposition \ref{prop:trans-2}]
Following Casdagli's result in \cite{Casdagli1986} combined with \cite[Proposition 3.9]{Yessen2011}, we have: for all $\mf{q}$ with $\abs{\mf{q}}$ sufficiently large, $\gamma_{(1,\mf{q})}^*$ intersects the center-stable manifolds transversally, and this intersection occurs on a compact segment along $\gamma_{(1,\mf{q})}^*$. Hence all sufficiently small perturbations of $\gamma_{(1,\mf{q})}^*$ intersect the center-stable manifolds transversally.

Similarly, combination of results in \cite{Dahl2010}  with \cite[Proposition 3.9]{Yessen2011} shows that for all $\mf{p}$ with $\abs{\mf{p}-1}$ sufficiently large, $\gamma_{(\mf{p},0)}^*$ intersects the center-stable manifolds transversally, so again all sufficiently small perturbations of $\gamma_{(\mf{p},0)}^*$ also intersect the center-stable manifolds transversally.
\end{proof}
Combination of Propositions \ref{prop:trans}, \ref{prop:trans-1} and \ref{prop:trans-2} gives the proof of Theorem \ref{thm_main1}.
%
%
%
\begin{figure}[t]
\centering
\setlength{\unitlength}{0.5mm}
\begin{picture}(120,120)

\put(10,10){\framebox(100,100)}

\put(5,3){(0,0)}

\put(50,3){(0, 1/2)}

\put(102,3){(1,0)}

\put(5,113){(0,1)}

\put(102,113){(1,1)}

\put(8,8){$\bullet$}

\put(108,8){$\bullet$}

\put(58,8){$\bullet$}

\put(8,108){$\bullet$}

\put(108,108){$\bullet$}

\put(10,10){\line(161,100){36.2}}

\put(60,10){\line(-61,100){13.8}}

\put(40,16){$1$}

\linethickness{0.5mm}

\put(11,10){\line(1,0){50}}

\linethickness{0.075mm}

\put(60,10){\line(161,100){36.2}}

\put(110,10){\line(-61,100){13.8}}

\put(90,16){$1$}

\put(60,110){\line(-161,-100){36.2}}

\put(10,110){\line(61,-100){13.8}}

\put(25,100){$1$}

\put(110,110){\line(-161,-100){36.2}}

\put(60,110){\line(61,-100){13.8}}

\put(75,100){$1$}

\put(10,60){\line(61,-100){22.1}}

\put(18,28){$5$}

\put(10,60){\line(161,100){22.1}}

\put(10,110){\line(61,-100){22.1}}

\put(18,78){$5$}

\put(110,60){\line(-61,100){22.1}}

\put(102,38){$5$}

\put(110,60){\line(-161,-100){22.1}}

\put(110,10){\line(-61,100){22.1}}

\put(102,88){$5$}

\put(60,60){\line(-161,-100){36.3}}

\put(60,60){\line(161,100){36.3}}

\put(60,60){\line(-61,100){22.1}}

\put(60,60){\line(61,-100){22.1}}

\put(88,46.35){\line(-161,-100){14.3}}

\put(32,73.65){\line(161,100){14.3}}

\put(37.88,46.28){\line(61,-100){8.4}}

\put(82.12,73.72){\line(-61,100){8.4}}

\put(34,34){$4$}

\put(84,84){$4$}

\put(34,84){$3$}

\put(84,34){$3$}

\put(34,60){$2$}

\put(84,60){$2$}

\put(60,34){$6$}

\put(60,84){$6$}

\end{picture}
\caption{The Markov partition for
$T|_{\mathbb{S}}$ (picture taken from \cite{Damanik2009}).}\label{fig:Casdagli-Markov}
\end{figure}
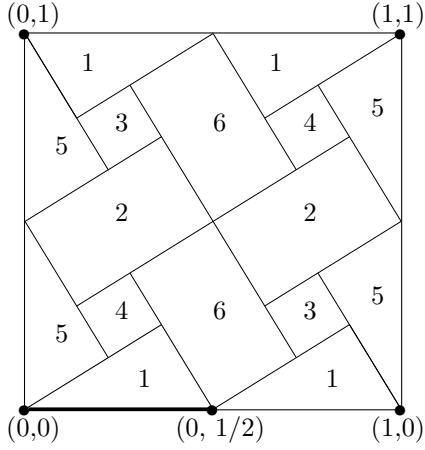
%
%
%

\subsection{Proof of theorem \ref{thm_main2}}\label{sec:proof_main2}

For the existence of the limit in \eqref{eq:thm_main2-1}, it is enough to prove the following
\begin{prop}\label{prop:thm-main2}
There exists $C > 0$ and for every $n\in\N$ there exists  a subset $U_n$ of $\Sigma_{(\mf{p},\mf{q})}$ of full $d\mathcal{N}$-measure, such that for all $E\in U_n$, we have
\begin{align}\label{prop:thm-main2-eq1}
\limsup_{\epsilon\downarrow 0}\frac{\log\mathcal{N}(E-\epsilon, E+\epsilon)}{\log\epsilon}\leq C,
\end{align}
with $C$ independent of $n$, and
\begin{align}\label{prop:thm-main2-eq2}
\limsup_{\epsilon\downarrow 0}\frac{\log\mathcal{N}(E-\epsilon, E+\epsilon)}{\log\epsilon} - \liminf_{\epsilon\downarrow 0}\frac{\log\mathcal{N}(E-\epsilon, E+\epsilon)}{\log\epsilon} \leq \frac{1}{n}.
\end{align}
\end{prop}
\begin{proof}[Proof of Proposition \ref{prop:thm-main2}]
Transversal intersection of $\gamma_{(\mf{p},\mf{q})}^*$ with the center-stable manifolds will be the main ingredient for us; however, we have proved transversality in only special cases. On the other hand, we know that tangential intersections, if such exist, occur at no more than finitely many points. Since $d\mathcal{N}$ is non-atomic and our results are stated modulo a set of measure zero, we may exclude those points. We also exclude the extreme upper and lower boundary points of the spectrum, as these may correspond to intersection of $\gamma_{(\mf{p},\mf{q})}^*$ with $S_0$; while this doesn't present great complications, it is certainly more convenient to work away from $S_0$.

For what follows, the interested reader should see \cite{Damanik2011} for technical details where we omit them. 

Under $\gamma: \R\rightarrow\R^3$ from \eqref{gamma_defn}, the spectrum for the pure Hamiltonian, $\Sigma_{(1,0)}$, corresponds to the line in $\R^3$ connecting the points $P_1$ and $P_2$. Following the convention that we've established above, call this line segment $\overline{\gamma_{(1,0)}^*}$. A Markov partition for $\mathcal{A}$ on $\T^2$ is shown in Figure \ref{fig:Casdagli-Markov}. The preimage of $\overline{\gamma_{(1,0)}^*}$ under $F$ from \eqref{eq:ptrans-1} is the line segment $l\equiv [0,1/2]\times\set{0}$ in $\T^2$ (i.e. the segment connecting $(0,0)$ and $(0,1/2)$ in Figure \ref{fig:Casdagli-Markov}). Let $\mathcal{R}$ be the element of the Markov partition containing $l$. Take the Lebesgue measure on $\mathcal{R}$, normalize it, project it onto $l$, and push the resulting measure forward under $F$ onto $\overline{\gamma_{(1,0)}^*}$. The resulting probability measure on $\overline{\gamma_{(1,0)}^*}$, denoted by $d\widetilde{\mathcal{N}}_0$, corresponds to the density of states measure for the pure Hamiltonian, which we denote by $d\mathcal{N}_0$, under the identification
\begin{align}\label{eq:measure-ident}
\gamma_{(1,0)}: E\mapsto\left(\frac{E}{2}, \frac{E}{2}, 1\right).
\end{align}
Now, let $\set{\rho_V^1, \rho_V^2} = S_V\cap\mathrm{Per}_2(f)$. Observe that $\rho_V^1 = \rho_V^2$ if and only if $V = 0$. For $i = 1, 2$ and $V > 0$, $\rho_V^i$ is a hyperbolic fixed point for $f_V^2$ on $S_V$. The stable manifolds to $\rho_V^i$, $\set{W^s(\rho_V^1)}_{V > 0}$ and $\set{W^s(\rho_V^2)}_{V > 0}$, foliate two two-dimensional injectively immersed submanifolds of $\R^3$ that connect smoothly along $W^\mathrm{ss}(P_1)$ to form $W^\mathrm{cs}(P_1)$ (see \cite[Theorem B]{Pugh1997} for details).

Now fix $(\mf{p},\mf{q})\neq (1,0)$. Define a probability measure $\mu$ on $\gamma_{(\mf{p},\mf{q})}^*$ as follows. Let $(\beta_1(t),\beta_2(t))$ be a smooth regular curve in $\R^2$ with $(\beta_1(0),\beta_2(0)) = (1,0)$, $(\beta_1(1),\beta_2(1)) = (\mf{p},\mf{q})$. Denote by $W$ the smooth two-dimensional submanifold of $\R^3$ given by 
\begin{align*}
W := \bigcup_{t\in[0,1]}\gamma_{(\beta_1(t),\beta_2(t))}^*.
\end{align*}
For $t\in[0,1]$, even if $\gamma_{(\beta_1(t),\beta_2(t))}^*$ intersects $W^\mathrm{cs}(P_1)$ tangentially (at finitely many points), this intersection cannot be quadratic (this would produce an isolated point), nor can an intersection contain connected components (since the set of intersections is a Cantor set). It follows that $W\cap W^\mathrm{cs}(P_1)$ consists of uncountably many smooth regular curves, each with one endpoint in $\gamma_{(1,0)}^*$, and the other in $\gamma_{(\mf{p},\mf{q})}^*$. Hence a holonomy map from $\gamma_{(\mf{p},\mf{q})}^*\cap W^\mathrm{cs}(P_1)$ to $\gamma_{(1,0)}\cap W^\mathrm{ss}(P_1)$, given by projection along these curves (this map is not one-to-one), is well-defined; call this map $\mathcal{H}$. Now, with $E_0, E_1\in \gamma_{(\mf{p},\mf{q})}^*\cap W^{cs}(P_1)$, let the interval bounded by $E_0, E_1$ carry the same weight under $\mu$ as the interval bounded by $\mathcal{H}(E_0)$ and $\mathcal{H}(E_1)$ carries under $d\mathcal{N}_0$. This defines $\mu$ on intervals with endpoints in a dense subset, and hence completely determines $\mu$. 
\begin{claim}\label{claim:prop-help1}
The measure $d\mathcal{N}_{(\mf{p},\mf{q})}$ corresponds to the measure $\mu$ under the identification \eqref{eq:measure-ident}.
\end{claim}
\begin{proof}[Proof of Claim \ref{claim:prop-help1}]
Take two distinct points $E_0, E_1\in \gamma^{-1}_{(1,0)}(\overline{\gamma_{(1,0)}^*}\cap W^\mathrm{ss}(P_1))$. As soon as the parameters $\mf{p}, \mf{q}$ are turned on, a gap opens at the points $E_0, E_1$. Let $I$ be the interval bounded by $E_0$ and $E_1$, and $I_{(\mf{p},\mf{q})}$ the interval bounded by the two gaps. Then $d\mathcal{N}_{(\mf{p},\mf{q})}(I_{(\mf{p},\mf{q})}) = d\mathcal{N}_0(I)$. On the other hand, $d\mathcal{N}_0(I)$ is, modulo \eqref{eq:measure-ident}, just $d\widetilde{\mathcal{N}}_0(\gamma_{(1,0)}(I))$, which is the same as $\mu(\gamma_{(\mf{p},\mf{q})}(I_{(\mf{p},\mf{q})}))$.
\end{proof}
Let us now concentrate on $\mu$ along $\overline{\gamma_{(\mf{p},\mf{q})}^*}$. Let $\Gamma$ denote the intersection of $\gamma_{(\mf{p},\mf{q})}^*$ with the center-stable manifolds, excluding points of tangential intersection and those corresponding to the extreme boundary points of the spectrum.

Say $m\in\Gamma\cap S_{V_m}$, $V_m > 0$. With the notation from Lemma \ref{lem:holder}, let $\tau^*_m$ be a compact arc along $\tau_{V_m}^*$ containing $m$ in its interior and short enough such that the holonomy map $h$ restricted to $\tau^*_m$ is H\"older with exponent $\alpha$, as in Lemma \ref{lem:holder}. We may assume that the endpoints of $\tau_m$ lie on the center-stable manifolds. A slight modification of results in \cite{Damanik2011} gives
\begin{lem}\label{lem:dg2011}
There exists a measure $\mu_m$ defined on $\tau_m^*$, whose topological support is the intersection of $\tau_m^*$ with the center-stable manifolds, with the following properties. If $E_0, E_1$ are distinct points in $\tau_m^*\cap W^\mathrm{cs}(P_1)$ which are not boundary points of the same gap, and if $\widetilde{E}_0, \widetilde{E}_1\in\overline{\gamma_{(1,0)}^*}$ such that $E_i$ is a boundary point of the gap that opens at $\widetilde{E}_i$, then the interval bounded by $E_0, E_1$ carries the same weight under $\mu_m$ as does the interval bounded by $\widetilde{E}_0,\widetilde{E}_1$ under $d\widetilde{N}_0$. Moreover, for $\mu_m$-almost every $x\in\tau_m^*$, we have
\begin{align}\label{eq:lem-dg2011-1}
\lim_{\epsilon\downarrow 0}\frac{\log\mu_m B_{\tau_m^*, \epsilon}(x)}{\log\epsilon} = d(m)\in \R,
\end{align}
with
\begin{align}\label{eq:lem-dg2011-2}
0 < \inf_{m\in\Gamma}\set{d(m)}\text{,\hspace{4mm}}\sup_{m\in\Gamma}\set{d(m)}<\infty.
\end{align}
Moreover,
\begin{align}\label{eq:lem-dg2011-3}
\lim_{(\mf{p},\mf{q})\rightarrow(1,0)}\inf_{m\in\Gamma}\set{d(m)}= \lim_{(\mf{p},\mf{q})\rightarrow(1,0)}\sup_{m\in\Gamma}\set{d(m)} = 1.
\end{align}
Here $B_{\tau_m^*, \epsilon}(x)$ denotes $\epsilon$-ball around $x$ along $\tau_m^*$.
\end{lem}
As an immediate consequence, if $E_0, E_1\in\tau_m^*$ in the domain of $h$, then the interval bounded by $E_0, E_1$ carries the same weight under $\mu_m$ as does the interval bounded by $h(E_0), h(E_1)$ under $\mu$. As a consequence of \eqref{eq:lem-dg2011-1} and \eqref{eq:lem-dg2011-2} together with $\alpha$-H\"older continuity of $h$, we have the following. For $\mu_m$-almost every $x\in \tau_m^*$ in the domain of $h$,
\begin{align}\label{eq:prop-exactness}
\alpha d(m)\leq \liminf_{\epsilon\downarrow 0}\frac{\log\mu B_{\overline{\gamma^*}, \epsilon}(h(x))}{\log\epsilon}\leq \limsup_{\epsilon\downarrow 0}\frac{\log\mu B_{\overline{\gamma^*}, \epsilon}(h(x))}{\log\epsilon}\leq \frac{1}{\alpha}d(m),
\end{align}
which implies
\begin{align*}
\limsup_{\epsilon\downarrow 0}\frac{\log\mu B_{\overline{\gamma^*}, \epsilon}(h(x))}{\log\epsilon} - \liminf_{\epsilon\downarrow 0}\frac{\log\mu B_{\overline{\gamma^*}, \epsilon}(h(x))}{\log\epsilon}\leq \left(\frac{1}{\alpha}-\alpha\right)\sup_{m\in\Gamma}\set{d(m)}.
\end{align*}
Now choose $\alpha\in(0,1)$ such that
\begin{align*}
\left(\frac{1}{\alpha} - \alpha\right)< \frac{1}{n\sup_{m\in\Gamma}{d(m)}}.
\end{align*}
Let $\mf{V}_m$ be the subset of $\tau_m^*$ of full $\mu_m$-measure for which the conclusion of Lemma \ref{lem:dg2011} holds, and set $U_n = \bigcup_{m\in\Gamma}h(\mf{V}_m)$. Finally, apply Claim \ref{claim:prop-help1}.
\end{proof}
That the limit in \eqref{eq:thm_main2-1} is strictly positive follows from \eqref{eq:lem-dg2011-2}, and \eqref{eq:thm_main2-2} follows from \eqref{eq:lem-dg2011-3}.  It remains to prove \eqref{eq:thm_main2-3}.

From \cite{Damanik2011} we have that $d(m) < \frac{1}{2}\hdim(\Omega_{V_m})$, where $\Omega_{V_m}$ is the non-wandering set for $f_{V_m}$ on $S_{V_m}$. On the other hand, we have $\lhdim(\Sigma_{(\mf{p},\mf{q})}, m) = \frac{1}{2}\hdim(\Omega_{V_m})$. Also, from \cite{Damanik2011} we know that $d(m)$ depends continuously on $m$ (in fact it is the restriction to $\Gamma$ of a smooth function), so there is $\delta > 0$ such that for all $m\in\Gamma$, $\lhdim(\Sigma_{(\mf{p},\mf{q})},m)\geq d(m) + \delta$. Thus combined with \eqref{eq:prop-exactness} we have
\begin{align*}
\limsup_{\epsilon\downarrow 0}\frac{\log\mu B_{\overline{\gamma^*}, \epsilon}(h(x))}{\log\epsilon} \leq \frac{1}{\alpha}(\lhdim(\Sigma_{(\mf{p},\mf{q})}, m) - \delta).
\end{align*}
On the other hand local Hausdorff dimension is a continuous function over the spectrum, hence, assuming $x$ and $m$ are sufficiently close (that is, assuming $x\in\tau_m^*$ with $\tau_m^*$ sufficiently short), we have
\begin{align*}
\limsup_{\epsilon\downarrow 0}\frac{\log\mu B_{\overline{\gamma^*}, \epsilon}(h(x))}{\log\epsilon} \leq \frac{1}{\alpha}\left(\lhdim(\Sigma_{(\mf{p},\mf{q})}, h(x))-\frac{\delta}{2}\right).
\end{align*}
We can take $\alpha$ arbitrarily close to one. Now \eqref{eq:thm_main2-3} follows.

\section{Concluding remarks and open problems}\label{conclusion}

We believe that Theorem \ref{thm_main_result1} holds in greater generality. Namely, we believe that $\overline{\gamma^*_{(\mf{p},\mf{q})}}$ intersects the center-stable manifolds transversally for all $(\mf{p},\mf{q})\neq (1,0)$, $\mf{p}\neq 0$. This would allow one to extend many results that are currently known for the diagonal and the off-diagonal operators (e.g. \cites{Damanik2010,Damanik2011}). We should mention, however, that even in those two cases, transversality isn't known for all values of $\mf{q}$ and $\mf{p}$, respectively (compare \cites{Casdagli1986, Damanik2009}).

\begin{conj} With the notation as above, for all $(\mf{p},\mf{q})\neq (1,0)$, $\mf{p}\neq 0$, $\overline{\gamma^*_{(\mf{p},\mf{q})}}$ intersects the center-stable manifolds transversally.
\end{conj}

We also note that, unlike in the diagonal and the off-diagonal cases, there are parameters $(\mf{p},\mf{q})$ for which the spectrum of the corresponding tridiagonal operator has full Hausdorff dimension, contrary to what one would expect from previous results.

Another particularly curious problem is analyticity of the Hausdorff dimension. We believe this to be true:

\begin{conj}\label{conj:anal} If $\alpha(t) = (\mf{p}(t),\mf{q}(t))$ is an analytic curve in $\R^2\setminus{(1,0)}$ and $\mf{p}(t)\neq 0$ for all $t$, then $\hdim(\Sigma_{\alpha(t)})$ is analytic as a function of $t$.
\end{conj}

In fact, this ties in with the monotonicity problem for the diagonal (and similarly the off-diagonal) model:

\begin{conj}\label{conj:mon} The Hausdorff dimension of the spectrum of the diagonal operator is a monotone-decreasing function of $\mf{q}\in [0,\infty)$.
\end{conj}

Let us show how the claim of Conjecture \ref{conj:anal} follows from the conclusion of Conjecture \ref{conj:mon}. 

Take $\alpha$ as in the statement of Conjecture \ref{conj:anal}. Let $V(t)$ be such that for the lower endpoint of the spectrum $\Sigma_{\alpha(t)}$, which we denote by $l(t)$, we have $\gamma(l(t))\in S_{V(t)}$, where $\gamma$ is the curve that was defined in \eqref{gamma_defn}. Clearly $V(t)$ is analytic. On the other hand, monotonicity of $\mf{q}\mapsto \hdim(\Sigma_{(1,\mf{q})})$ implies monotonicity of $V\mapsto \hdim(\Omega_V)$ (see \eqref{eq_local_dim}). Thus by monotonicity, we have:
\begin{align*}
\hdim(\Sigma_{\alpha(t)}) = \lhdim(\Sigma_{\alpha(t)}, l(t)) = \frac{1}{2}\hdim(\Omega_{V(t)}).
\end{align*}
Since $t\mapsto V(t)$ is analytic, and $V\mapsto\hdim(\Omega_V)$ is also analytic, analyticity of $\hdim(\Sigma_{\alpha(t)})$ follows. 

We should also that strict  upper and lower bounds on the Hausdorff dimension of $\Omega_V$, as a function of $V$, have been given in \cite{Damanik2010} for all $V$ sufficiently close to zero.

Evidently both conjectures would follow from monotonicity of $V\mapsto\hdim(\Omega_V)$.

\section*{Acknowledgement}

I wish to thank my thesis advisor, Anton Gorodetski, for the invaluable support and guidance.

I also wish to thank Svetlana Jitomirskaya, David Damanik and Christoph Marx for useful discussions.

\bibliographystyle{plain}
\bibliography{../../../bibliography}
\end{document}